\documentclass{amsart}
\usepackage{color}
\usepackage[utf8]{inputenc}
\usepackage{amsmath,amssymb,amsfonts,amstext,amsthm}
\usepackage{url}
\usepackage{graphicx}
\input xy
\xyoption{all}
\usepackage[all]{xy}

\graphicspath{{images/}{../images/}}
\usepackage{subfiles}
\usepackage{tikz}

\newcommand{\nwc}{\newcommand}

\nwc{\aaa}{\mathcal{F}}
\nwc{\aap}{\mathcal{F}_{P}}
\nwc{\al}{\alpha}

\nwc{\C}{\mathbb{C}}
\nwc{\cb}{\overline{C}}
\nwc{\ccc}{\mathfrak{c}}
\nwc{\ch}{\widehat{C}}
\nwc{\cin}{\textbf{(v)}}
\nwc{\cl}{C'}
\nwc{\cp}{\mathcal{C}_{P}}
\nwc{\cpll}{\mathfrak{c}_{P'}}
\nwc{\ct}{\widetilde{C}}

\nwc{\dd}{\mathcal{L}}
\nwc{\ddd}{\mathfrak{d}}
\nwc{\ddl}{\mathcal{L}'}
\nwc{\dlp}{\delta_{P}}
\nwc{\doi}{\textbf{(ii)}}

\nwc{\enq}{$$}
\nwc{\He}{H^{\ast}}

\nwc{\fl}{\flushleft}
\nwc{\fff}{\mathcal{F}}
\nwc{\ffp}{\mathcal{F}_{P}}
\nwc{\ffq}{\mathcal{F}_{Q}}
\nwc{\ffl}{\mathcal{F}'}

\nwc{\G}{\mathcal{G}}
\nwc{\Ga}{\Gamma}
\nwc{\gtl}{\widetilde{g}}

\nwc{\hra}{\hookrightarrow}
\nwc{\hua}{h^{1}(C,\aaa )}

\nwc{\kk}{{\rm K}}

\nwc{\llb}{\mathcal{L}}

\nwc{\mb}{\mathbb}
\nwc{\mc}{\mathcal}
\nwc{\mm}{\mathfrak{m}}
\nwc{\mmp}{\mathfrak{m}_{P}}
\nwc{\mpd}{\mathfrak{m}_{P}^{2}}

\nwc{\nn}{\mathbb{N}}

\nwc{\ob}{\overline{\mathcal{O}}}
\nwc{\obr}{\mathcal{O}^*}
\nwc{\obp}{\overline{\mathcal{O}}_P}
\nwc{\och}{\mathcal{O}_{\hat{C}}}
\nwc{\oh}{\hat{\mathcal{O}}}
\nwc{\ohp}{\hat{\mathcal{O}}_{P}}
\nwc{\ol}{\mathcal{O}'}
\nwc{\oma}{\Omega (\mathfrak{a})}
\nwc{\omo}{\Omega (\mathcal{O})}
\nwc{\oo}{\mathcal{O}}
\nwc{\op}{\mathcal{O}_P}
\nwc{\opc}{\mathcal{O}_{P,C}}
\nwc{\oph}{\hat{\mathcal{O}}_{P}}
\nwc{\opl}{\mathcal{O}_{P}'}
\nwc{\oplc}{\mathcal{O}_{P,C}'}
\nwc{\opll}{\mathcal{O}_{P'}}
\nwc{\opt}{\tilde{\mathcal{O}}_{P}}
\nwc{\optt}{{\mathcal{O}}_{\tilde{P}}}
\nwc{\oq}{\mathcal{O}_{Q}}
\nwc{\oqt}{\tilde{\mathcal{O}}_{Q}}
\nwc{\ot}{\widetilde{\mathcal{O}}}
\nwc{\overop}{\bar{\oo}_{P}}

\nwc{\pb}{\overline{P}}
\nwc{\pbb}{P^*}
\nwc{\pbi}{\overline{P_{i}}}
\nwc{\pbr}{\overline{P_{r}}}
\nwc{\pgmd}{\mathbb{P}^{g+2}}
\nwc{\pgmu}{\mathbb{P}^{g+1}}
\nwc{\ph}{\hat{P}}
\nwc{\pp}{\mathbb{P}}
\nwc{\prv}{\noindent\textbook{Proof}:}
\nwc{\pt}{\widetilde{P}}
\nwc{\ptl}{\tilde{P}}
\nwc{\pum}{\mathbb{P}^{1}}

\nwc{\qh}{\hat{Q}}
\nwc{\qtl}{\tilde{Q}}
\nwc{\qua}{\textbf{(iv)}}

\nwc{\rh}{\hat{R}}

\nwc{\sei}{\textbf{(vi)}}
\nwc{\sep}{\beq\ast\ \ast\ \ast\enq}
\nwc{\sig}{\sigma}
\nwc{\Sig}{\Sigma}
\nwc{\ssp}{S_{P}}
\nwc{\sss}{{\rm S}}

\nwc{\tre}{\textbf{(iii)}}

\nwc{\um}{\textbf{(i)}}

\nwc{\vpb}{v_{\overline{P}}}
\nwc{\vtxp}{\widetilde{V}_{x,P}}
\nwc{\vxp}{V_{x,P}}

\let \wt=\widetilde

\nwc{\wh}{\hat{\omega}}
\nwc{\whp}{\hat{\omega}_{P}}
\nwc{\woch}{\omega\cdot\mathcal{O}_{\hat{C}}}
\nwc{\woh}{\omega\cdot\hat{\mathcal{O}}}
\nwc{\ww}{\omega}
\nwc{\wwb}{\omega^*}
\nwc{\wwct}{\omega _{\widetilde{C}}}
\nwc{\wwh}{\widehat{\omega}}
\nwc{\wwhp}{\widehat{\omega}_P}
\nwc{\wwp}{\omega _{P}}
\nwc{\wwt}{\widetilde{\omega}}
\nwc{\wwtp}{\widetilde{\omega}_P}

\nwc{\zz}{\mathbb{Z}}

\newtheorem{coro}{Corollary}[section]

\newtheorem{lemma}[coro]{Lemma}

\newtheorem{prop}[coro]{Proposition}

\newtheorem{prob}[coro]{Problem}
\newtheorem{rem}[coro]{Remark}

\newtheorem{thm}[coro]{Theorem}
\newtheorem{conj}[coro]{Conjecture}

\newtheorem{const}[coro]{Construction}

\newtheorem{claim}[coro]{Claim}
\let \fl=\flushleft

\let \ga=\gamma
\let \sub=\subset

\let \al=\alpha
\let \pr=\prime

\let \ov=\overline

\begin{document}

\title{Weight bounds for $(3,\ga)$-hyperelliptic curves}

\author{Rafael Barbosa da Silva}
\address{Instituto de Matem\'atica, Universidade Federal Fluminense, Rua Prof Waldemar de Freitas, S/N, Campus do Gragoat\'a, CEP 24.210-201, Niter\'oi, RJ, Brazil}
\email{rafaelufrpe@gmail.com}

\author{Ethan Cotterill}
\address{Instituto de Matem\'atica, Universidade Federal Fluminense, Rua Prof Waldemar de Freitas, S/N, Campus do Gragoat\'a, CEP 24.210-201, Niter\'oi, RJ, Brazil}
\email{cotterill.ethan@gmail.com}

\begin{abstract}
{\it $(N,\ga)$-hyperelliptic} semigroups were introduced by Fernando Torres to encapsulate the most salient properties of Weierstrass semigroups associated to totally-ramified points of $N$-fold covers of curves of genus $\ga$. Torres characterized $(2,\ga)$-hyperelliptic semigroups of maximal weight whenever their genus is large relative to $\ga$. Here we do the same for $(3,\ga)$-hyperelliptic semigroups, and we formulate a conjecture about the general case whenever $N \geq 3$ is prime. 
\end{abstract}
\maketitle

\section{Canonical inflection, and Weierstrass weights of $(3,\ga)$-hyperelliptic curves}

The {\it inflection} of a linear series $(L,V)$ on a smooth complex algebraic curve $C$ is a fundamental numerical invariant of $(L,V)$. To define it, assume that $L$ is a line bundle of degree $d$, while $V \sub H^0(C,L)$ is an $(r+1)$-dimensional subspace of holomorphic sections. We then let
\begin{equation}\label{inflection_divisor}
I_{(L,V)}:= \sum_{p \in C} (a_i(p)-i)p
\end{equation}
in which $a_0(p) < a_1(p) < \dots < a_r(p)$ is the sequence of {\it vanishing orders} of sections of $V$ in $p$. Thus $I_{(L,V)}$ is a divisor on $C$, and a fundamental result of Pl\"ucker establishes that its total degree is uniquely determined by the values of $d$, $r$, and the genus $g$ of $C$. On the other hand, the {\it local} inflection indices $(a_i(p)-i)$ in \eqref{inflection_divisor} are less predictable, and it is interesting to see how they vary as a function of the global geometry of $(L,V)$.

\medskip
The case of the {\it canonical} series $|K_C|= (K_C, H^0(C,K_C))$ is distinguished. Indeed, inflection points of the canonical series are {\it Weierstrass points} and have been widely studied. Serre duality relates the set $\rm{A}$ of vanishing orders in $p$ of sections of the canonical series with the set $\rm{S}$ of pole orders in $p$ of meromorphic functions on $C$: we have $\rm{A}= (\mb{N} \setminus \rm{S})-1$. Here $\rm{S}={\rm S}(C,P)$ is the {\it Weierstrass semigroup} of $C$ in $p$; its {\it weight} is defined to be $\sum_{i=1}^g (\ell_i-i)$, where $\ell_1 < \ell_1 < \dots < \ell_g$, $\ell_i=a_i(p)+1$ is the sequence of elements belonging to $\mb{N} \setminus \rm{S}$. A natural problem in this context is to explain, on geometric grounds, how to maximize the weight of the Weierstrass semigroup ${\rm S}(C,p)$, or equivalently the inflection of the canonical series $|K_C|$ in $p$.

\medskip
We are specifically interested in this problem in cases in which the curve $C$ may be realized as an $N$-fold branched cover $\pi: C \rightarrow B$ of a curve $B$ of genus $\ga$. We further require that $\pi$ have at least one totally ramified point, $p$; then via pullback of meromorphic functions, we see that $N \cdot {\rm S}(B,q) \sub {\rm S}(C,p)$. C\'icero Carvalho and Fernando Torres studied such Weierstrass semigroups ${\rm S}(C,p)$ from a combinatorial point of view, 
calling them {\it $(N,\ga)$-hyperelliptic} \cite{CT} by analogy with hyperelliptic semigroups (which correspond to the case $N=2$, $\ga=0$). Their salient properties are that 
\begin{enumerate}
\item[1.] The first $\ga$ positive elements $n_1, \dots, n_{\ga}$ of $\rm S$ are multiples of $N$, and $n_{\ga}=2\ga N$; and
\item[2.] $(2\ga+1)N$ belongs to $\rm S$.
\end{enumerate}

\subsection{$(3,\ga)$-hyperelliptic semigroups}
Now let ${\rm S} \sub \mb{N}$ denote a numerical semigroup of genus $g$. Recall that this means that the complement $G_{\rm S}= \mb{N} \setminus {\rm S}$ is of cardinality $g$; say
\[
G_{\rm S}= \{\ell_1, \dots, \ell_g\}
\]
where $\ell_i < \ell_j$ whenever $1 \leq i< j \leq g$.
We define the {\it weight} of ${\rm S}$ to be the quantity
\[
W_{\rm S}:= \sum_{i=1}^{g} (\ell_i-i).
\]

\begin{rem}\label{rem0}
The conductor of a numerical semigroup of genus $g$ is always at most $2g$; consequently, a numerical semigroup ${\rm S}$ of genus $g$ is uniquely determined by its truncation ${\rm S}^*:= {\rm S} \cap [2g]$. It follows that two numerical semigroups ${\rm S}$ and ${\rm T}$ of the same genus $g$ satisfy 
\[
W_{\rm S} \geq W_{\rm T} \iff \sum_{s \in {\rm S}^*} s \leq \sum_{t \in {\rm T}^*} t.
\]
\end{rem}
We will denote $\sum_{s \in {\rm S}^*} s$ by $\wt{I}_S$; up to recalibration by a constant, this is the {\it inflection} of ${\rm S}$ (not to be confused with the inflection of the canonical series of the preceding subsection!), and we will refer to it as such.

\medskip
Now let $N,\ga \geq 0$ be integers. Following \cite{CT}, we say that ${\rm S}$ is {\it $(N,\ga)$-hyperelliptic} if it satisfies the following conditions:
\begin{enumerate}
\item[1.] The first $\ga$ positive elements $n_1, \dots, n_{\ga}$ of $\rm S$ are multiples of $N$, and $n_{\ga}=2\ga N$; and
\item[2.] $(2\ga+1)N$ belongs to $\rm S$.
\end{enumerate}

It will be useful for our purposes to filter $(N,\ga)$-hyperelliptic semigroups according to their residues modulo $N$. Given an integer $0 \leq i \leq N-1$, we let ${\rm S}_i$ denote the subset of ${\rm S}^{\ast}$ of elements with residue $i$ modulo $N$. Given an integer $M \geq 2$, we use $(u)_{M}$ as a shorthand for $u \text{ mod }M$.

\medskip
Torres showed in \cite{To2} that $(2,\ga)$-hyperelliptic semigroups of genus $g \geq 2\ga$ satisfy the characteristic weight inequality
\begin{equation}\label{two_bound}
\binom{g-2\ga}{2}\leq W_{\rm S}\leq \binom{g-2\ga}{2} + 2\ga^2.
\end{equation}

In an attempt to generalize \eqref{two_bound}, Carvalho and Torres 
determined a lower bound on the weight of $(N,\ga)$-hyperelliptic semigroups of genus $g$.

\begin{thm}[\cite{CT}, Thm 4.1]\label{(N,gamma)_lb}
Let $\ga$, $g$, and $N$ be non-negative integers. Every $(N,\ga)$-hyperelliptic semigroup ${\rm S}$ of genus $g$ satisfies
\[
W_{\rm S} \geq \frac{(g-N\ga-N+1+q)(g-N\ga-q)}{2(N-1)}.
\]
where $q= (g-\ga)_{(N-1)}$.
\end{thm}
 
The lower bound in Theorem~\ref{(N,gamma)_lb}, which is sharp, is easy to prove using the fact that the weight is given by the number of squares in the complement inside a $2g \times 2g$ grid of squares of the Dyck path associated to ${\rm S}^*$ in \cite{BdM}. An upper bound, however, is significantly more delicate. In this note, we will find a upper bound on the weight of $(N,\ga)$-hyperelliptic semigroups of genus $g$ when $N=3$.

\medskip
Accordingly, let $\rm S$ be a $(3,\ga)$-hyperelliptic semigroup of genus $g$. In light of Remark~\ref{rem0}, in order to maximize the weight of $\rm S$, it suffices to minimize the inflection $\wt{I}_S$. 
Carvalho and Torres observed
that when the genus $g$ is {\it sufficiently small} relative to $\ga$, the following scheme produces the truncation ${\rm S}^*$ of an $(3,\ga)$-hyperelliptic semigroup ${\rm S}$ of maximal weight relative to $g$. 

\begin{const}\label{gsmall}Let g and $\ga$ be nonnegative integers. Let ${\rm S}^{ct}$ denote the unique completion of the $g$-element subset ${\rm S}^{ct,*}= {\rm S}^{ct}_0 \sqcup {\rm S}^{ct}_1 \sqcup {\rm S}^{ct}_2 \sub [2g]$ to a numerical semigroup of genus $g$, where
\[
\begin{split}
{\rm S}^{ct}_0&:=\{6i: i=1,\dots,\ga\} \sqcup \{6\ga+3j: 1 \leq j \leq \bigg \lfloor \frac{2g}{3} \bigg \rfloor-2\ga\}, \\
{\rm S}^{ct}_r&:= (u_1+S_0)\cap [2g], \text{ and}\\
\rm S^{ct}_{3-r}&:=(\rm S_r+\rm S_r)\cap [2g].
\end{split}
\]
Here we assume $g \not\equiv 2 \text{ (mod } 3)$; $u_1=g-3\ga +1
$ is the smallest possible non-3-divisible element of $[2g]$ for which $\#(\rm S^{ct}_0 \sqcup \rm S^{ct}_1 \sqcup \rm S^{ct}_2)=g$; and $r=(g+1)_{3}
$. 
\end{const}

A straightforward calculation shows the Carvalho--Torres semigroup ${\rm S}^{ct}$ has weight
\[
\frac{A(5A+1)}{6}+ 3A\ga+ 6\ga^2+ 3\ga- \binom{g+1}{2}.
\]
where $A=g-3\ga$. When $g$ is large, however, Construction~\ref{gsmall} fails to produce a $(3,\ga)$-hyperelliptic semigroup of maximal weight.

\begin{thm}\label{maximal_weight_conjecture} 
Assume that $g \gg \ga$. The $(3,\ga)$-hyperelliptic semigroup ${\rm S}^{bc}$ of maximal weight obtained by completing the $g$-element subset ${\rm S}^{bc,*}= {\rm S}^{bc}_0 \sqcup {\rm S}^{bc}_1 \sqcup {\rm S}^{bc}_2$ to a numerical semigroup of genus $g$ is such that $6$ belongs to ${\rm S}^{bc}_0$, while ${\rm S}^{bc}_1$ and ${\rm S}^{bc}_2$ are specified by the following scheme. 

\begin{enumerate}
    \item If  $g-\ga \not \equiv 2$ (\text{mod }3), then 
    \[
    \begin{split}
    {\rm S}^{bc}_{(u_1)_{3}}&=\{u_1+               3k: k=0, \dots, \tau\} \text{ and }\\
    {\rm S}^{bc}_{3-(u_1)_{3}}&=({\rm S}^{bc}_{(u_1)_{3}}+{\rm S}^{bc}_{(u_1)_{3}}) \cap [2g]=\{2u_1+3k: k=0, \dots, \nu\}.
    \end{split}
    \]
    \item If  $g-\ga \equiv 2$ (\text{mod }3), then
    \[
    \begin{split}
    {\rm S}^{bc}_{(u_1)_{3}}&=\{u_1\} \sqcup \{u_1+6+3k: k=0, \dots, \tau\} \text{ and} \\
    {\rm S}^{bc}_{3-(u_1)_{3}}&=({\rm S}^{bc}_{(u_1)_{3}}+{\rm S}^{bc}_{(u_1)_{3}}) \cap [2g]=\{2u_1\} \sqcup \{2u_1+6+3k: k=0, \dots, \nu\}. 
    \end{split}
    \]
\end{enumerate}
Here $u_1=g-\ga+1-2\lfloor\frac{(g-\ga)_{3}}{2}\rfloor$; 
 $\tau:= \lceil\frac{g}{3}\rceil+\lfloor\frac{\ga}{3}\rfloor-1$ whenever $\lfloor\frac{(\ga)_{3}}{2}\rfloor\neq (u_1)_{3}$ and $\tau:= \lceil\frac{g}{3}\rceil+\lfloor\frac{\ga}{3}\rfloor$ otherwise; while 
 $\nu:= \ga-\lfloor\frac{\ga}{3}\rfloor-1$ whenever $\lfloor\frac{(\ga)_{3}}{2}\rfloor\neq (u_1)_{3}$ and 
 $\nu:= \ga-\lfloor\frac{\ga}{3}\rfloor-2$ otherwise. 
 In particular, we have
\begin{itemize}
    \item $\#{\rm S}^{bc}_0=\lfloor\frac{2g}{3}\rfloor-\ga$;
    \item $\#{\rm S}^{bc}_{(u_1)_{3}}=\lceil\frac{g}{3}\rceil+\lfloor\frac{\ga}{3}\rfloor$ when $\lfloor\frac{(\ga)_{3}}{2}\rfloor\neq (u_1)_{3}$, and $\#{\rm S}^{bc}_{(u_1)_{3}}=\lceil\frac{g}{3}\rceil+\lfloor\frac{\ga}{3}\rfloor+1$ otherwise; and
    \item $\#{\rm S}^{bc}_{3-(u_1)_{3}}=\ga-\lfloor\frac{\ga}{3}\rfloor$ when $\lfloor\frac{(\ga)_{3}}{2}\rfloor\neq (u_1)_{3}$, and $\#{\rm S}^{bc}_{3-(u_1)_{3}}=\ga-\lfloor\frac{\ga}{3}\rfloor-1$ otherwise. 
\end{itemize}
\end{thm}

A straightforward calculation shows that the weight of the semigroup of Theorem~\ref{maximal_weight_conjecture} is equal to $\frac{B(5B+1)}{6}+3\ga^2 -\binom{g+1}{2}$ whenever $g-\ga \not\equiv 2 \text{ (mod }3)$, and $\frac{B(5B-3)}{6}+3\ga^2+6- \binom{g+1}{2}$ otherwise, where $B=g-\ga$.

\subsection{Roadmap} The core of this paper is {\bf Section~\ref{maximal_weight_conj_proof}}, in which we prove Theorem~\ref{maximal_weight_conjecture}. To do, we argue in three stages. In {\bf Step 1}, we show that being of maximal weight imposes strong conditions on the value of minimal insertion $u_1$ beyond $3\langle 2,2\ga+1\rangle$, and on the cardinalities of the residue sets ${\rm S}_i$. We begin by assuming that $u_1$ itself is fixed; in Proposition~\ref{maxT1} we establish that maximal-weight semigroups always are associated with residue sets ${\rm S}_{(u_1)_3}$ of maximal size. We subsequently allow $u_1$ to vary. In Lemmas~\ref{T1contained} and \ref{T1form} we show, roughly speaking, that maximal-weight semigroups are associated with minimal insertions that lie between $g-3\ga+1$, the minimum possible value, and $g-\ga+1$, the value for which $\#{\rm S}_{(u_1)_3}$ is (uniformly) maximized. In Proposition~\ref{maxT1v2}, we show that $\#{\rm S}_{(u_1)_3}$ {\it must} be maximal; doing so ensures that the difference between ${\rm S}_{3-(u_1)_3}$ and the ``quadratic" sum ${\rm S}_{(u_1)_3}+{\rm S}_{(u_1)_3}$ is {\it minimized}.

\medskip
{\fl In} {\bf Step 2}, we carry out a case-by-case comparison of all semigroups associated with a fixed choice of ${\rm S}_0$, and for which ${\rm S}_{(u_1)_3}$ is maximal. Using Claim~\ref{aux3} (and variations thereof, all of which involve finer analyses of how ${\rm S}_{(u_1)_3}+{\rm S}_{(u_1)_3}$ sits inside of ${\rm S}_{3-(u_1)_3}$), we obtain a finite list of candidate maximal-weight semigroups ${\rm S}$.



\medskip
{\fl In} {\bf Step 3}, we use Proposition~\ref{chooseT0} to show that the ${\rm S}$ of maximal weight are precisely those given in the statement of Theorem~\ref{maximal_weight_conjecture}. In particular, $6$ necessarily belongs to ${\rm S}_0$, which amounts to the statement that ${\rm S}$, if it is indeed geometrically realizable, arises from a cover of a hyperelliptic curve marked in a hyperelliptic Weierstrass point.

\medskip
{\fl {\bf Section~\ref{geometric_realization}}} is of a more speculative nature. Our Conjecture~\ref{maximal_weight_conj} explicitly predicts the structure of an $(N,\ga)$-hyperelliptic semigroup of maximal weight, whenever $N \geq 3$ is prime, $g \gg \ga$ and some additional divisibility conditions are satisfied by $N$, $g$, and $\ga$. In Proposition~\ref{buchweitz_for_S_N}, we show that the semigroups ${\rm S}_N$ in question, which generalize the numerical semigroup of maximal weight given in Theorem~\ref{maximal_weight_conjecture}, satisfy a well-known necessary criterion for geometric realizability due to Buchweitz. The upshot is that ascertaining the realizability of maximal-weight $(N,\ga)$-hyperelliptic semigroups in general is a delicate question.

\section{Proof of Theorem~\ref{maximal_weight_conjecture}}\label{maximal_weight_conj_proof}
Let ${\rm H}=H(g,\ga)$ denote the set of $(3,\ga)$-hyperelliptic semigroups of genus $g$. Remark~\ref{rem0} shows that in order to maximize the $\rm S$-weight of $\rm T \in {\rm H}$, it suffices to minimize the inflection $\wt{I}_{\rm T}$. 
To this end, let ${\rm T}^\ast:= {\rm T}_0 \sqcup {\rm T}_1 \sqcup {\rm T}_2$, where $\rm T_i=\{t\in {\rm T}^\ast| (t)_3=i\}$, and let $\ov{\rm T}_{0}:=\rm T_0\sqcup{\{0\}}$. We define
\[
u_1(\rm T):=\text{ min }\{t \in\rm T| (t)_3\neq 0\} \text{, }
u_1^H:=\text{ min }\{u_1(\rm T)| \rm T \in {\rm H}\}\text{ and }
\chi ({\rm T})=u_1({\rm T})-u_1^H.
\]

\begin{rem}\label{u1_remarks} Given a reference semigroup ${\rm S}$ of genus $g$, let
${\rm H}^{{\rm S}_0}
:=\{{\rm T\in \rm H}|{\rm T}_0={\rm S}_0\}$. 
A straightforward genus calculation shows that $u_1^H= g-3\ga+1+\lfloor\frac{(g)_3}{2}\rfloor$. \footnote{Cf. \cite[Lemma 3.2]{CT}, in which this explicit formula is replaced by an inequality that omits the perturbation term $\lfloor\frac{(g)_3}{2}\rfloor$.}
\end{rem}

\subsection{Step 1}
For a fixed choice of ${\rm S_0}$, we will characterize the semigroup in ${\rm H}^{{\rm S}_0}$ with minimal inflection. To do so, it is useful to further filter semigroups according to their associated values of $u_1=u_1({\rm S})$. Accordingly, we set
\[
{\rm H}_{u_1}:= \{{\rm T}\in {\rm H}^{{\rm S}_0}: u_1({\rm T})=u_1\}.
\]

\begin{lemma}\label{T2_bound}
Given ${\rm T}\in {\rm H}_{u_1}$, we have $\# \rm{T}_{3-(u_1)_3} \leq \ga+ \lceil \frac{\chi}{3} \rceil$, where $\chi:= u_1-u_1^H= u_1-g+3\ga-1-\lfloor\frac{(g)_3}{2}\rfloor$.
\end{lemma}  
\begin{proof}
Since the cardinality of $\rm{T}_{0}$ is fixed, the cardinality of $\rm{T}_{3-(u_1)_3}$ is maximized when $\# {\rm T}_{(u_1)_3}$ is minimized, i.e., when ${\rm T}_{(u_1)_3}=u_1+\ov{\rm T}_{0}$. We then have
${\rm T}^{\ast}= {\rm T}_0 \sqcup (u_1+\ov{{\rm T}}_{0}) \sqcup {\rm T}_{3-(u_1)_3}$, and as $(u_1+\ov{\rm T}_{0})$ fails to be an arithmetic sequence by precisely $\ga$ ``gaps" in $[2g]$ that are translations by $u_1$ of the $\ga$ multiples of 3 not contained in ${\rm T}_0$, it follows that
\[
\begin{split}
\#\rm{T}_{3-(u_1)_3}&=g-(\#{\rm T}_0+\#(u_1+\ov{\rm T}_{0})) \\
&= g-\bigg(\frac{2g-(2g)_3}{3}-\ga+ \frac{(2g-u_1)-(2g-u_1)_3}{3}+1-\ga \bigg) \\
&=-\frac{g}{3}+ 2\ga-1+ \frac{(2g)_3+(2g-u_1)_3+ u_1}{3} \\
&=\ga+ \frac{\chi}{3}+ \frac{(-g)_3+ (g+2-\lfloor\frac{(g)_3}{2}\rfloor -\chi)_3-2}{3} \\
&\leq \ga+ \bigg \lceil \frac{\chi}{3} \bigg \rceil.
\end{split}
\]
\end{proof}
\begin{lemma}\label{boundchi}
Assume $g \gg \ga$. 
We have $\chi<6\ga$ for any $(3,\ga)$-hyperelliptic semigroup of maximal weight.
\end{lemma}
\begin{proof}
Let {\rm T} be a semigroup such that $\chi:=u_1({\rm T})-u_1^H\geq 6\ga$; then $u_1=u_1({\rm T})\geq g+3\ga+1$. 
We will show that there exists a semigroup ${\rm S} \in {\rm H}^{{\rm T}_0}$ for which $\wt{I}_{S}<\wt{I}_{T}$. 

\medskip
To this end, fix ${\rm S}_0:={\rm T}_0$, $u_1^{\pr}=u_1({\rm S}):=u_1({\rm T})-(\chi-(\chi)_3)$, and set ${\rm S}_{(u_1^{\pr})_3}:=u_1^{\pr}+\ov{{\rm S}}_0$. Note that if $(\chi)_3=0$, we have $u^{\pr}_1=u_1^H$; it follows necessarily that ${\rm S}={\rm S}^{ct}$ as in Construction~\ref{gsmall}, and ${\rm S}_{3-(u_1^{\pr})_3}=({\rm S}_{(u_1^{\pr})_3}+{\rm S}_{(u_1^{\pr})_3})\cap[2g]$. On the other hand, if $(\chi)_3\neq0$, we have $u^{\pr}_1=u_1^H+(u_1^H)_3$ and
\[
\#{\rm S}_{u^{\pr}_1}=\bigg\lceil\frac{g}{3}\bigg\rceil-\bigg\lceil\frac{(g)_3}{2}\bigg\rceil \text{, which forces } \#{\rm S}_{3-(u^{\pr}_1)_3}=\ga+\bigg\lceil\frac{(g)_3}{2}\bigg\rceil.
\]
However in this case $\#({\rm S}_{(u_1^{\pr})_3}+{\rm S}_{(u_1^{\pr})_3})\cap[2g]=\ga-1$, so ${\rm S}_{3-(u^{\pr}_1)_3} \cap [2g]$ contains $1+\lceil\frac{(g)_3}{2}\rceil$ additional elements $y_1, \dots, y_{1+\lceil\frac{(g)_3}{2}\rceil}$ that belong to the complement of $({\rm S}_{(u_1^{\pr})_3}+{\rm S}_{(u_1^{\pr})_3})\cap[2g]$.

{\fl On} the other hand, from $u_1=u^{\pr}_1+(\chi-(\chi)_3)$ it follows that \[
\#(u_1+\ov{\rm T}_0)=\#{\rm S}_{(u^{\pr}_1)_3}-\frac{(\chi-(\chi)_3)}{3}.
\]
Note that $\#(u_1+\ov{\rm T}_0)$ fails to be an arithmetic sequence by precisely $\ga$ ``gaps" in $[2g]$. Now let $J=\ga+\frac{\chi-(\chi)_3}{3}$, and let $\{x_1,...,x_{J}\}$ denote the subset containing all elements of ${\rm T}_{(u_1)_3}\setminus(u_1+\ov{\rm T}_0)$ together with (if necessary) the smallest elements of ${\rm T}_{3-(u_1)_3}$; we then have $x_i> u_1$ for every $i$. There are $\#{\rm S}_{3-(u^{\pr}_1)_3}+\ga$ remaining elements in ${\rm T}_{3-(u_1)_3}$, and the smallest of these, call it $y$, is at least $2g-6\ga-3s_2+3$.
Indeed, otherwise $\#(y+\ov{\rm T}_0)\cap[2g]>s_2+\ga$; moreover, the largest $s_2$ elements of $y+\ov{\rm T}_0$ form an arithmetic sequence, and the $k$th among these is greater than or equal to $(2g-3)s_2-3k$, where $k=0,\dots,s_2-1$.
{\fl All} of the pieces are now in place to show that $\wt{I}_{\rm T}-\wt{I}_{\rm S}$ is positive. 
As ${\rm T}_0={\rm S}_0$, it suffices to show that
$(\wt{I}_{{\rm T}_{(u_{1})_3}}+\wt{I}_{{\rm T}_{3-(u_1)_3}})-(\wt{I}_{{\rm S}_{(u^{\pr}_1)_3}}+\wt{I}_{{\rm S}_{3-(u^{\pr}_1)_3}})> 0$,
where $\wt{I}_{{\rm T}_i}$ denotes the contribution of ${\rm T}_i$ to the inflection of the semigroup ${\rm T}$. Letting $s_1=\#{\rm S}_{(u^{\pr}_1)_3}$, $t_1=\#{\rm T}_{(u_1)_3}$, $s_2=\#{\rm S}_{3-(u^{\pr}_1)_3}$ and $t_2=\#{\rm T}_{3-(u_1)_3}$, we find that

{\small
\[
\begin{split}
\wt{I}_{{\rm S}_{(u^{\pr}_1)_3}}+\wt{I}_{{\rm S}_{3-(u^{\pr}_1)_3}}&\leq\ga u^{\pr}_1+\sum_{i=1}^{\ga-1}n_i
+\bigg(\frac{\chi-(\chi)_3-6\ga}{3}\bigg)(u^{\pr}_1+6\ga)+3\sum_{i=1}^{\delta}i +2\ga(u^{\pr}_1+\chi-(\chi)_3) \\
&+3\sum_{i=1}^{2\ga-1}i
+R(u^{\pr}_1+\chi-(\chi)_3+6\ga)+3\sum_{i=1}^{R-1}i+2gs_2-3\sum_{i=1}^{s_2-1}i
\end{split}
\]
}
where $\delta=\frac{\chi-(\chi)_3-6\ga}{3}$ and $R=s_1-3\ga-\delta$ and
{\small
\[
\begin{split}
\wt{I}_{{\rm T}_{(u_{1})_3}}+\wt{I}_{{\rm T}_{3-(u_1)_3}}&\geq\ga u_1+\sum_{i=1}^{\ga-1}n_i+R(u_1+6\ga)+3\sum_{i=1}^{R-1}i+\ga u_1 +\bigg(\frac{\chi-(\chi)_3-6\ga}{3}\bigg)u_1 \\
&+3\sum_{i=1}^{\delta}i+\ga(2g-6\ga-3s_2+3)+3\sum_{i=1}^{\ga-1}i+s_2(2g-3)-3\sum_{i=1}^{s_2-1}i.
\end{split}
\]
}
{\fl It} follows that
{\small
\begin{equation}\label{lemma2.3ineq}
\begin{split}
(\wt{I}_{{\rm T}_{(u_{1})_3}}+\wt{I}_{{\rm T}_{3-(u_1)_3}})-(\wt{I}_{{\rm S}_{(u^{\pr}_1)_3}}+\wt{I}_{{\rm S}_{3-(u^{\pr}_1)_3}})&\geq \delta(\chi-(\chi)_3-6\ga)+\ga(g-3\ga-3s_2+2)-3s_2-3\sum_{i=\ga}^{2\ga-1}i.
\end{split}
\end{equation}
}
{\fl However}, as $s_2 \leq \ga+1$, the right-hand side of \eqref{lemma2.3ineq} is bounded below by
\[
\delta(\chi-(\chi)_3-6\ga)+\ga(g-15\ga-10)-6
\]
which is strictly positive whenever $g\gg\ga$.
\end{proof} 

\begin{prop}\label{maxT1}
Assume $g \gg \ga$. Suppose that semigroups ${\rm S}$ and ${\rm T}$ belong to ${\rm H}_{u_1}$, with $\chi:=u_1-u_1^H<6\ga$ and $\# \rm{T}_{(u_1)_3}>\# \rm{S}_{(u_1)_3}$. Then $W_{\rm T}>W_{\rm S}$.
\end{prop}

\begin{proof}
Begin by noting that $u \geq 2g-1-(6\ga+\chi)$ for all $u \in {\rm T}_{3-(u_1)_3}$. Indeed, the presence of an element $u \leq 2g-1-(6\ga+\chi)-3$ in ${\rm T}_{3-(u_1)_3}$ would force 
\[
\{u+n_1,u+n_2,\dots,u+6\ga\}\sqcup\{u+6\ga +3,u+6\ga +6,\dots,u+6\ga+\chi+3\} \subset {\rm T}_{3-(u_1)_3}
\]
where $n_i$, $i \geq 1$ denotes the $i$th nonzero element of ${\rm T}$. But then $\# {\rm T}_{3-(u_1)_3}\geq \ga+ \frac{\chi}{3}+1$, which contradicts Lemma~\ref{T2_bound} above.

\medskip
On the other hand, we have $\{u_1+6\ga+3k: k \geq 0\} \cap [2g] \subset u_1+\ov{\rm T}_0$, by construction. It follows that $u\leq u_1+6\ga-3$ for all $u \in {\rm T}_{(u_1)}\setminus(u_1+\ov{\rm T}_0)$.
Note that, since $\chi<6\ga$, the difference
\[
(2g-1-(6\ga+\chi))-(u_1+6\ga-3)=g-9\ga-2\chi+1
\]
is positive whenever $g$ is large relative to $\ga$; it follows that the smallest element of ${\rm T}_{3-(u_1)_3}$ is larger than the largest element of ${\rm T}_{(u_1)}\setminus (u_1+\ov{\rm T}_0)$, and we conclude immediately. 
\end{proof}
The upshot of Proposition~\ref{maxT1} is that it suffices to consider semigroups ${\rm T} \in {\rm H}_{u_1}$ for which ${\rm T}_{(u_1)_3}$ is of maximal size.

\subsection{Step 1, bis} We now compare the inflection of two semigroups in ${\rm H}^{{\rm S}_0}$ associated with distinct values of $u_1$.
\begin{lemma}\label{T1contained}
Assume $g \gg \ga$. Let ${\rm S}$ and ${\rm T}$ be the semigroups such that
$\rm{S}_{(u_1)_3}$ and $\rm{T}_{(u^{\pr}_1)_3}$ are of maximal size in ${\rm H}_{u_1}$ and ${\rm H}_{u^{\prime}_1}$, respectively. If $\chi:=u_1-u_1^H<6\ga$ and ${\rm  S}_{(u_1)_3}\subsetneq{\rm T}_{(u^{\pr}_1)_3}$, then $\wt{I}_T < \wt{I}_S$.
\end{lemma}
\begin{proof}
By construction, we have $\{u_1+6\ga+3k|k\geq 0\}\subset {\rm T}_{(u^{\pr}_1)_3}\cap{\rm S}_{(u_1)_3}$. Thus, every $u\in {\rm T}_{(u^{\pr}_1)_3}\setminus{\rm S}_{(u_1)_3}$ satisfies $u\leq u_1+6\ga-3$.
On the other hand, as we saw in the proof of Lemma~\ref{maxT1}, every element of $\rm{S}_{3-(u_1)_3}$ is greater than or equal to $2g-1-(6\ga+\chi)$.
Moreover, as $\chi<6\ga$ it follows that the difference
\[
(2g-1-(6\ga+\chi))-(u_1+6\ga-3)=g-9\ga-2\chi+1
\]
is positive whenever $g$ is large relative to $\ga$. 
It follows that the smallest element of ${\rm S}_{3-(u_1)_3}$ is larger than the largest element of ${\rm T}_{(u^\pr_1)_3}\setminus {\rm S}_{(u_1)_3}$, and we conclude immediately.
\end{proof}
\begin{lemma}\label{T1form}
Assume that $g\gg\ga$ and let ${\rm T}$ 
be a $(3,\ga)$-hyperelliptic semigroup for which ${\rm T}_{(u_1)_3}=(u_1+3\mathbb{N})\cap[2g]$. Then $u_1=u_1({\rm T})\geq g-\ga+1$; moreover, whenever $u_1=g-\ga+1+\lfloor\frac{(g-\ga)_3}{2}\rfloor$, the cardinality of the residue set ${\rm T}_{(u_1)_3}$ is maximal among {\it all} $(3,\ga)$-hyperelliptic semigroups in ${\rm H}={\rm H}(g,\ga)$.
\end{lemma}
\begin{proof} 
Let ${\rm T}$ denote a semigroup for which ${\rm T}_{(u_1)_3}=(u_1+3\mathbb{N})\cap[2g]$. 
For such a semigroup, we have
\[
\begin{split}
\#{\rm T}_0&= \frac{2g}{3}- \frac{(2g)_3}{3} -\gamma, \# {\rm T}_{(u_1)_3}= \frac{2g-u_1}{3}+ 1- \frac{(2g-u_1)_3}{3}, \text{ and also} \\
\#{\rm T}_{3-(u_1)_3}&\geq\#({\rm T}_{(u_1)_3}+{\rm T}_{(u_1)_3})\cap[2g]= \frac{2g-2u_1}{3}+ 1- \frac{(2g-2u_1)_3}{3}
\end{split}
\]
because 
$({\rm T}_{(u_1)_3}+{\rm T}_{(u_1)_3})\cap[2g]\subset{\rm T}_{3-(u_1)_3}$
. It follows that
\[
\begin{split}
g&= \# {\rm T}_0+ \# {\rm T}_{(u_1)_3}+ \# {\rm T}_{3-(u_1)_3} \\
&\geq\frac{2g}{3}- \frac{(2g)_3}{3} -\ga+ \frac{2g-u_1}{3}+ 1- \frac{(2g-u_1)_3}{3}+ \frac{2g-2u_1}{3}+ 1- \frac{(2g-2u_1)_3}{3}.
\end{split}
\]
In other words, we have
\[
u_1\geq g -\ga +2 - \frac{(2g)_3+(2g-u_1)_3+(2g-2u_1)_3}{3}\geq g-\ga+1
\]
in which the second inequality follows from the fact that $u_1$ has nonzero residue modulo 3.

{\fl Note} that whenever $(g-\ga)_3=2$, the preceding inequality may be strengthened to 
$u_1\geq g-\ga+2$, since in that case $g-\ga+2$ is the smallest number with nonzero 3-residue greater than or equal to $g-\ga+1$. As
$\#{\rm T}_{(u_1)_3}= \lfloor\frac{2g-u_1}{3}\rfloor+1$, we obtain  
\begin{equation}\label{t1_equal}
\#{\rm T}_{(u_1)_3}=\bigg \lfloor\frac{g+\ga+2}{3} \bigg \rfloor \text{ if } (g-\ga)_3\neq2, \text { and }
\#{\rm T}_{(u_1)_3}=\bigg \lfloor\frac{g+\ga+1}{3} \bigg \rfloor \text{ otherwise}. 
\end{equation}
In the former (resp., latter) case, we have 
$\#({\rm T}_{(u_1)_3}+{\rm T}_{(u_1)_3})\cap[2g]=\#{\rm T}_{3-(u_1)_3}$ (resp., $\#({\rm T}_{(u_1)_3}+{\rm T}_{(u_1)_3})\cap[2g]=\#{\rm T}_{3-(u_1)_3}-1$); in particular,
\begin{equation}\label{boundT1T1}
\#({\rm T}_{(u_1)_3}+T_{(u_1)_3})\cap[2g]\geq \#{\rm T}_{3-(u_1)_3}-1.
\end{equation}
{\fl To prove} the second part of the lemma we will apply the following observation.

\begin{claim}\label{aux}
Let $g$ be an integer and let $S=\{x_1,x_2,\dots\}$ be an increasing 
sequence of nonnegative integers for which $S\cap[g]=\{x_1,x_2,\dots,x_k\}$ for some $k \geq 1$. 
Then $\#(S+S)\cap[2g]\geq2k-1$.
\end{claim}

\begin{proof}[Proof of Claim~\ref{aux}]
Let $x_1:=x$ and let $x_{i+1}:=x+a_i$ for $i=1,\cdots,k-1$. The $2k-1$ distinct elements $2x,2x+a_1,2x+a_2,\cdots,2x+a_{k-1},2x+a_{k-1}+a_1,2x+a_{k-1}+a_2,\cdots,2x+2a_{k-1}$ all belong to $\#(S+S)\cap[2g]$.
\end{proof}
{\fl Now} suppose there exists a semigroup ${\rm S}\in {\rm H}^{{\rm T}_0}$ for which $\#{\rm S}_{(u^{\pr}_1)_3}\geq\#{\rm T}_{(u_1)_3}+1$, where $u_1^{\pr}:=u_1({\rm S})$.
Then \eqref{t1_equal} implies that
\begin{equation}\label{s1_ineq}
\#{\rm S}_{(u^{\pr}_1)_3}\geq \bigg \lfloor\frac{g+\ga+2}{3} \bigg \rfloor+1 \text{ if } (g-\ga)_3 \neq2, \text{ and }
\#{\rm S}_{(u^{\pr}_1)_3}\geq \bigg \lfloor\frac{g+\ga+1}{3} \bigg \rfloor+1 \text{ otherwise}. 
\end{equation}
Note that
\[
{\rm S}_{(u^{\pr}_1)_3}=({\rm S}_{(u^{\pr}_1)_3}\cap[g])\sqcup({\rm S}_{(u^{\pr}_1)_3}\cap\{g+1,\dots,2g\})
\]
{\fl and} $\#{\rm S}_{(u^{\pr}_1)_3}\cap\{g+1,g+2,\cdots,2g\}\leq\lfloor\frac{2g-(g+1)}{3}\rfloor+1=\lfloor\frac{g+2}{3}\rfloor$. Applying \eqref{s1_ineq}, we deduce that
\begin{equation}\label{s1_in_g}
\begin{split}
\#{\rm S}_{(u^{\pr}_1)_3}\cap[g]&\geq \bigg \lfloor\frac{g+\ga+2}{3} \bigg \rfloor+1-\bigg \lfloor\frac{g+2}{3} \bigg \rfloor=\bigg \lfloor\frac{\ga}{3}\bigg \rfloor+1
\text{ if } (g-\ga)_3 \neq2, \text { and }\\
\#{\rm S}_{(u^{\pr}_1)_3}\cap[g]&\geq \bigg \lfloor\frac{g+\ga+1}{3} \bigg \rfloor+1- \bigg \lfloor\frac{g+2}{3} \bigg \rfloor=\bigg \lfloor\frac{\ga}{3} \bigg \rfloor+2 \text{ otherwise}. 
\end{split}
\end{equation}

{\fl Applying} Claim~\ref{aux} in tandem with the inequalities \eqref{s1_in_g} now yields
{\small
\begin{equation}\label{s1s1_ineq}
\#{({\rm S}_{(u_1^{\pr})_3}+{\rm S}_{(u_1^{\pr})_3})\cap[2g]}\geq2 \bigg \lfloor\frac{\ga}{3} \bigg \rfloor+1 \text{ if } g-\ga\neq2, \text { and }
\#{({\rm S}_{(u_1^{\pr})_3}+{\rm S}_{(u_1^{\pr})_3})\cap[2g]}\geq2 \bigg \lfloor\frac{\ga}{3} \bigg \rfloor+3 \text{ otherwise}. 
\end{equation}
}

{\fl Using} the fact that $\#{\rm S}_0=\lfloor\frac{2g}{3}\rfloor-\ga$ and summing (this together with) the inequalities \eqref{s1_ineq} and \eqref{s1s1_ineq}, we deduce that 
{\small
\[
\begin{split}
\#{\rm S}_0+\#{\rm S}_{u^{\pr}_1}+\#{({\rm S}_{(u_1^{\pr})_3}+{\rm S}_{(u_1^{\pr})_3})\cap[2g]} \geq& \bigg \lfloor\frac{2g}{3} \bigg \rfloor-\ga+ \bigg \lfloor\frac{g+\ga+2}{3} \bigg \rfloor+1+ 2 \bigg \lfloor\frac{\ga}{3} \bigg \rfloor+1\geq g+1 \text{ if } (g-\ga)_3\neq2, \text { and }\\
\#{\rm S}_0+\#{\rm S}_{u^{\pr}_1}+\#{({\rm S}_{(u_1^{\pr})_3}+{\rm S}_{(u_1^{\pr})_3})\cap[2g]} \geq& \bigg \lfloor\frac{2g}{3}\bigg \rfloor-\ga+ \bigg \lfloor\frac{g+\ga+1}{3} \bigg \rfloor+1 2 \bigg \lfloor\frac{\ga}{3} \bigg \rfloor+3\geq g+1 \text{ otherwise}. 
\end{split}
\]
}
Both conclusions are absurd, as they contradict the fact that ${\rm S}$ is of genus $g$.
\end{proof}

\begin{rem}\label{completeT2}
When there is equality in \eqref{boundT1T1}, so that $\#({\rm T}_{(u_1)_3}+T_{(u_1)_3})\cap[2g]=\#{\rm T}_{3-(u_1)_3}-1$, the smallest possible element that may be inserted to $(T_{(u_1)_3}+ T_{(u_1)_3})$ to produce ${\rm T}_{3-(u_1)_3}$ is $2u_1-n_1$, where $n_1$ is the smallest positive element of ${\rm T}$; in particular, if ${\rm T}$ is of maximal weight the insertion is necessarily $2u_1-n_1$. In particular, when referring to the semigroup ${\rm T}$ of Lemma~\ref{T1form} above, hereafter we may suppose that ${\rm T}_{3-(u_1)_3}=({\rm T}_{(u_1)_3}+T_{(u_1)_3})\cap[2g]\sqcup\{ 2u_1-n_1 \}$. 
\end{rem}

\begin{prop}\label{maxT1v2}
Assume that $g \gg \ga$. Let $u_1=g-\ga+1+\lfloor\frac{(g-\ga)_3}{2}\rfloor$ and let ${\rm T}$ be a semigroup for which ${\rm T}_{(u_1)_3}=(u_1+3\mathbb{N})\cap[2g]$. Any semigroup ${\rm S}$ for which 
$\#{\rm S}_{(u^{\pr}_1)_3}<\#{\rm T}_{(u_1)_3}$ satisfies $\wt{I}_T< \wt{I}_S$, where $u^{\pr}_1=u_1({\rm S})$. \end{prop}

\begin{proof}
As ${\rm T}_0={\rm S}_0$, it suffices to show that
\[
(\wt{I}_{S_{(u^{\pr}_1)_3}}-\wt{I}_{T_{(u_{1})_3}})+(\wt{I}_{S_{3-(u^{\pr}_1)_3}}-\wt{I}_{T_{3-(u_1)_3}})> 0.
\]
Accordingly, let $s_0=\#{\rm S_0}$, $t_0=\#{\rm T}_0$, $s_1=\#\rm S_{(u^{\pr}_1)_3}$, $t_1=\#\rm T_{(u_1)_3}$, $s_2=\#\rm S_{3-(u^{\pr}_1)_3}$ and $t_2=\#\rm T_{3-(u_1)_3}$.

From our initial hypotheses we have $u_1\leq g-\ga+2$; it follows that
\begin{equation}\label{T_inequalities}
\wt{I}_{T_{(u_1)_3}}\leq   t_1(g-\ga+2)+3\sum_{i=1}^{t_1-1}i \text{ and }
\wt{I}_{T_{3-(u_1)_3}}\leq  t_2(2g-2\ga+4) +3\sum_{i=1}^{t_2-1}i.
\end{equation}
Indeed, the first inequality in \eqref{T_inequalities} is clear, while 
the second inequality is clear whenever $\#({T_{(u_1)_3}}+{T_{(u_1)_3}})\cap[2g]=t_2$. 
On other hand, whenever $\#({T_{(u_1)_3}}+{T_{(u_1)_3}})\cap[2g]=t_2-1$, $({T_{(u_1)_3}}+{T_{(u_1)_3}})\cap[2g]$ is an arithmetic sequence with smallest element $2u_1$ that contains every element in ${\rm T}$ with residue equal to $(2u_1)_3$ between $2u_1$ and $2g$. Consequently, any element that is added to $({T_{(u_1)_3}}+{T_{(u_1)_3}})\cap[2g]$ in order to yield ${T_{3-(u_1)_3}}$ is necessarily less than $2u_1$. The right-hand side of the second inequality thus computes the sum of $t_2$ terms of an arithmetic sequence starting from $2g-2\ga+4$, of which the first $t_2-1$ terms are greater than or equal to (distinct) elements in $({T_{(u_1)_3}}+{T_{(u_1)_3}})\cap[2g]$ while the last is larger than any element that we may add to $({T_{(u_1)_3}}+{T_{(u_1)_3}})\cap[2g]$ in order to yield ${T_{3-(u_1)_3}}$. In particular, the second inequality in \eqref{T_inequalities} is strict whenever $\#({T_{(u_1)_3}}+{T_{(u_1)_3}})\cap[2g]=t_2-1$.


{\bf \flushleft {Case 1:}} Suppose that $(u_1^{\pr})_3=(u_1)_3$. If $u_1^{\pr}\geq u_1$, then ${\rm S}_{(u^{\pr}_1)_3}\subset{\rm T}_{(u_1)_3}$, and the result follows from Lemma~\ref{T1contained}. Thus without loss of generality we may assume that $u_1^{\pr}< u_1$. Note that $(u_1^{\pr}+6\ga+3\mathbb{N})\cap[2g]\subset {\rm S}_{(u^{\pr}_1)_3}\cap{\rm T}_{(u_1)_3}$. We set 
\[
{\rm T}_{1}^{\pr}={\rm T}_{(u_1)_3}\setminus((u_1^{\pr}+6\ga+3\mathbb{N})\cap[2g]) \text{ and } 
{\rm S}_{1}^{\pr}={\rm S}_{(u^{\pr}_1)_3}\setminus ((u_1^{\pr}+6\ga+3\mathbb{N})\cap[2g])
\]
and we let $t^{\pr}_1=\#{\rm T}_{1}^{\pr}$ and $s^{\pr}_1=\#{\rm S}_{1}^{\pr}
$ denote their respective cardinalities.
By construction, we have  $t_1^{\pr}-s^{\pr}_1=t_1-s_1$, 
$\wt{I}_{S_{(u^{\pr}_1)_3}}-\wt{I}_{T_{(u_{1})_3}}=\wt{I}_{S^{\pr}_{1}}-\wt{I}_{T^{\pr}_{1}}$ and it's easy to see that $\wt{I}_{T^{\pr}_{1}}\leq   t^{\pr}_1(g-\ga+2)+3\sum_{i=1}^{t^{\pr}_1-1}i$.

{\fl Remark~\ref{u1_remarks}} establishes that
$u_1^{\pr}\geq g-3\ga+1$; it follows that

\begin{equation}\label{I_S_pr}
\wt{I}_{S^{\pr}_{1}}\geq
s^{\pr}_1(g-3\ga+1)+3\sum_{i=1}^{s^{\pr}_1-1}i
\text{ and }
\wt{I}_{S_{3-(u^{\pr}_1)_3}}\geq s_2(2g-6\ga+2)+3\sum_{i=1}^{s_2-1}i.
\end{equation}

{\fl Thus}
\[
\begin{split}
\wt{I}_{S^{\pr}_{1}}-\wt{I}_{T^{\pr}_{1}}&\geq (s_1^{\pr}-t_1^{\pr})g-(3s_1^{\pr}-t_1^{\pr})\ga +s_1^{\pr}-2t_1^{\pr}
-3\sum_{i=s_1^{\pr}}^{t^{\pr}_1-1}i\text{ and}\\
\wt{I}_{S_{3-(u^{\pr}_1)_3}}-\wt{I}_{T_{3-(u_1)_3}}&\geq
2(s_2-t_2)g-2(3s_2-t_2)\ga +2(s_2-2t_2)
+3\sum_{i=t_2}^{s_2-1}i.
\end{split}
\]
Summing and bearing in mind that 
$t_1^{\pr}-s_1^{\pr}=s_2-t_2$, we find that
\begin{equation}\label{case_1_ineq}
\wt{I}_S-\wt{I}_T\geq
(s_2-t_2)g+(t_1^{\pr}+t_2)\ga+(s_1^{\pr}+2s_2)-3(s_1^{\pr}+2s_2)\ga-2(t_1^{\pr}+2t_2)+3\sum_{i=t_2}^{s_2-1}i -3\sum_{i=s_1^{\pr}}^{t^{\pr}_1-1}i.
\end{equation}

{\fl Since} $g\gg\ga$ by construction ${\rm S}_1^{\pr}$ and ${\rm T}_1^{\pr}$ are contained in $\{u_1^{\pr},u_1^{\pr}+3,...,u_1^{\pr}+6\ga-3\}$, so $s_1^{\pr}$ and $t_1^{\pr}$ are each at most $2\ga$. On the other hand, clearly $s_1^{\pr}\geq\ga$, because $\{u_1^{\pr},u_1^{\pr}+n_1,...,u_1^{\pr}+n_{\ga-1}\}$ belongs to ${\rm S}^{\pr}$ so $t_1^{\pr}-s_1^{\pr}\leq\ga$. Since $t_2=g-(t_0+t_1)\leq 2 \lfloor\frac{\ga}{3} \rfloor$ then $s_2=t_2+(t_1^{\pr}-s_1^{\pr})\leq \ga+2 \lfloor\frac{\ga}{3} \rfloor$. Thus the lower bound in \eqref{case_1_ineq} is positive whenever $g\gg \ga$.

{\bf \flushleft {Case 2:}} Suppose now that $(u_1^{\pr})_3\neq(u_1)_3$. In what follows, $r=(u_1)_3$.
Let ${\rm V}$ denote any semigroup for which $u^{\pr\pr}_1=u_1({\rm V})$ satisfies $u^{\pr\pr}_1=u_1+r$ and ${\rm V}_{(u^{\pr\pr}_1)_3}=(u^{\pr\pr}_1+3\mathbb{N})\cap[2g]$. Denote $v_0=\#{\rm V_0}$, $v_1=\#\rm S_{(u^{\pr\pr}_1)_3}$, and $v_2=\#\rm S_{3-(u^{\pr\pr}_1)_3}$. Because ${\rm V}_{(u^{\pr\pr}_1)_3}$ and ${\rm T}_{(u_1)_3}$ are arithmetic sequences in $[2g]$ in which all successive differences are 3, 
and $r$ is equal to 1 or 2, it follows that \[
\begin{split}
t_1-v_1&=
    \bigg(\frac{2g-u_1}{3}-\frac{(2g-u_1)_3}{3}+1\bigg)-\bigg(\frac{2g-(u_1+r)_3}{3}-\frac{(2g-(u_1+r))_3}{3}+1\bigg)\\
&=\frac{r}{3}-\frac{(2g-u_1)_3}{3}+\frac{(2g-(u_1+r))_3}{3}
\end{split}
\]
is equal to $0$ or $1$.
Similarly,
\[
\#({\rm T}_{(u_1)_3}+\#{\rm T}_{(u_1)_3})\cap[2g]-\#({\rm V}_{(u^{\pr\pr}_1)_3}+{\rm V}_{(u^{\pr\pr}_1)_3})\cap[2g]
\]
is equal to either $r$ or $r-1$. Just as in \eqref{boundT1T1}, we have $\#({\rm T}_{(u_1)_3}+\#{\rm T}_{(u_1)_3})\cap[2g]\geq t_2-1$, and it follows that \[
\#({\rm V}_{(u^{\pr\pr}_1)_3}+{\rm V}_{(u^{\pr\pr}_1)_3})\cap[2g]\geq t_2-r-1 .
\]

{\fl We} now compute
\[
\wt{I}_{V_{(u^{\pr\pr}_1)_3}}=v_1(u_1+r)+3\sum_{i=1}^{v_1-1}i
\text{ and }
\wt{I}_{V_{3-(u^{\pr\pr}_1)_3}}= (t_2-r-1)(2u_1+2r)+3\sum_{i=1}^{t_2-r-2}i+L
\]
in which the two first summands in the formula for $\wt{I}_{V_{3-(u^{\pr\pr}_1)_3}}$ are obtained by summing those $(t_2-r-1)$ smallest elements in $V_{3-(u^{\pr\pr}_1)_3}$ greater than or equal to (and including) $2u_1^{\pr\pr}$, and $L$ is the sum of the $\ell=v_2-(t_2-r-1)$ remaining elements $\{a_1,...,a_{\ell}\}$ of ${\rm V}_{3-(u^{\pr\pr}_1)_3}$. 

{\fl Note} that 
$a_i\geq2u_1^{\pr\pr}-n_{(l-i+1)}$ for every $i$, where $n_i$ is the $i-th$ positive element of ${\rm S}$. 
Indeed,
if $a_i<2u_1^{\pr\pr}-n_{(l-i+1)}$ for some $i$, it would follow that the $l+1$ distinct elements ${a_1,a_2,\cdots,a_i,a_i+n_1,a_i+n_{2},...,a_i+n_{l-i+1}}$ all belong to ${\rm V}_{3-(u_1^{\pr})_3}$ and are smaller than $2u_1^{\pr\pr}$, which would in turn force ${\rm V}_{3-(u_1^{\pr\pr})_3}$ to have cardinality strictly larger than $s_2$, which is absurd.

{\fl Summing} over $i$, we now obtain 
$L\geq 2u_1^{\pr\pr}(v_2-t_2+r+1)-\sum_{i=1}^{\ell} n_i$. Since $v_2-(t_2-r-1)\leq4$, it follows that
\begin{equation}\label{I_V2}
\wt{I}_{V_{3-(u^{\pr\pr}_1)_3}}\geq v_2(2u_1+2r)+3\sum_{i=1}^{t_2-r-2}i-\sum_{i=1}^{4}n_i.
\end{equation}

{\fl Therefore}
\[
\begin{split}
\wt{I}_{V_{(u^{\pr\pr}_{1})_3}}-\wt{I}_{T_{(u_1)_3}}&\geq rv_1-(t_1-v_1)u_1-3(t_1-v_1)v_1 \text{ and}\\
(\wt{I}_{V_{3-(u^{\pr\pr}_1)_3}}-\wt{I}_{T_{3-(u_1)_3}})&\geq
2rv_2+(v_2-t_2)2u_1 -3 \sum_{i=t_2-r-1}^{t_2-1}i-\sum_{i=1}^{4}n_i.
\end{split}
\]
Summing and applying the fact that 
$t_1-s_1=s_2-t_2$, we find that
\begin{equation}\label{last_estimate2}
\wt{I}_V-\wt{I}_T\geq r(v_1+2v_2)-(t_1-v_1)(3v_1-u_1)-3\sum_{i=t_2-r-1}^{t_2-1}i-\sum_{i=1}^{4}n_i.
\end{equation}
Since $r>0$ and $v_1\geq t_1-1=\lfloor\frac{g+\ga+2}{3} \rfloor-1$, the first summand on the right-hand side of \eqref{last_estimate2} is greater than 
$\lfloor\frac{g+\ga+2}{3} \rfloor$, while the other summands are bounded polynomially in $\ga$. 
Indeed, in the second summand we have $(t_1-v_1)\leq 1$, while 
\[
3v_1-u_1\leq 3\bigg\lfloor\frac{g+\ga+2}{3} \bigg\rfloor-g-\ga+1\leq 2\ga+1.
\]
Since $t_2=g-(t_0+t_1)\leq 2 \lfloor\frac{\ga}{3} \rfloor$ the third summand is bounded, and since $n_i\leq 6\ga$ for every $i$, the fourth and final summand is bounded too. Thus the lower bound in \eqref{last_estimate2} is positive whenever $g\gg \ga$.

\medskip
There are now three basic possibilities for $v_1=\#{\rm V}_{(u_1^{\pr \pr})_3}$ in relation to $s_1=\#{\rm T}_{(u_1^{\pr})_3}$ and $t_1=\#{\rm T}_{(u_1)_3}$. Namely, either $v_1=t_1$, in which case $s_1<v_1$; $v_1=t_1-1$ and $s_1<v_1$; or $v_1=t_1-1$ and $s_1=v_1$. In the first two cases, arguing as in Case 1 above yields $\wt{I}_{S}-\wt{I}_{V}>0$, and therefore $\wt{I}_{T}<\wt{I}_{S}$, whenever $g \gg \ga$. 

Finally, say that $v_1=t_1-1$ and $s_1=v_1$. We begin by calculating $\wt{I}_S-\wt{I}_V$. 
Since ${\rm V}_{(u_1^{\pr\pr})_3}$ is a arithmetic sequence and $(u_1^{\pr})_3=(u_1^{\pr\pr})_3$, the fact that $v_1=s_1$ forces $u_1^{\pr}\leq u_1^{\pr\pr}$. So, repeating the same argument as in Case 1, we find that $(u_1^{\pr}+6\ga+3\mathbb{N})\cap[2g]\subset {\rm S}_{(u^{\pr}_1)_3}\cap{\rm V}_{(u^{\pr\pr}_1)_3}$. We now set 
\[
{\rm V}_{1}^{\pr}:={\rm V}_{(u^{\pr\pr}_1)_3}\setminus((u_1^{\pr}+6\ga+3\mathbb{N})\cap[2g]) \text{ and } 
{\rm S}_{1}^{\pr}:={\rm S}_{(u^{\pr}_1)_3}\setminus ((u_1^{\pr}+6\ga+3\mathbb{N})\cap[2g])
\]
and we let $v^{\pr}_1=\#{\rm V}_{1}^{\pr}$ and $s^{\pr}_1=\#{\rm S}_{1}^{\pr}
$ denote their respective cardinalities.
By construction, we have  $v_1^{\pr}-s^{\pr}_1=v_1-s_1=0$ and
$\wt{I}_{S^{\pr}_{1}}-\wt{I}_{V^{\pr}_{1}}=\wt{I}_{S_{(u^{\pr}_1)_3}}-\wt{I}_{V_{(u^{\pr\pr}_{1})_3}}$. As $r\leq2$, it follows that $u_1^{\pr\pr}\leq g-\ga+4$ and consequently
\begin{equation}\label{I_V_ineq}
\begin{split}
\wt{I}_{V^{\pr}_{1}}\leq   v^{\pr}_1(g-\ga+4)+3\sum_{i=1}^{v^{\pr}_1-1}i\text{ and }
\wt{I}_{V_{3-(u^{\pr\pr}_1)_3}}\leq v_2(2g-2\ga+8)+3\sum_{i=1}^{v^{\pr}_2-1}i.
\end{split}
\end{equation}
Indeed, the first inequality in \eqref{I_V_ineq} is clear. To see that the second inequality holds, note that every insertion in $({V_{(u^{\pr\pr}_1)_3}}+{V_{(u^{\pr\pr}_1)_3}})\cap[2g]$ made to yield ${V_{3-(u_1^{\pr \pr})_3}}$ is necessarily less than $2u_1^{\pr \pr}$. The expression on the right-hand side of the second inequality thus computes the sum of $\#({V_{(u^{\pr\pr}_1)_3}}+{V_{(u^{\pr\pr}_1)_3}})\cap[2g]$ elements of an arithmetic sequence starting from $2g-2\ga+8$, together with $v_2-\#({V_{(u^{\pr\pr}_1)_3}}+{V_{(u^{\pr\pr}_1)_3}})\cap[2g]$ additional elements, each of which is larger than any of the insertions. In particular, the second inequality in \eqref{I_V_ineq} is strict.
{\fl Adding} \eqref{I_V_ineq} to \eqref{I_S_pr} (which describes $\wt{I}_{S^{\pr}_{1}}$ and $\wt{I}_{S_{3-(u^{\pr}_1)_3}}$) while bearing in mind that $v_1^{\pr}=s^{\pr}_1$ and $v_2=s_2$, it follows that
\[
\wt{I}_{S^{\pr}_{1}}-\wt{I}_{V^{\pr}_{1}}
\geq -s_1^{\pr}(2\ga+3)\text{ and }
\wt{I}_{S_{3-(u^{\pr}_1)_3}}-\wt{I}_{V_{3-(u^{\pr\pr}_1)_3}}
\geq
-2s_2(2\ga+3).
\]
In particular, we have
\begin{equation}\label{I_S_minus_I_V}
\wt{I}_S-\wt{I}_V\geq
-(s_1^{\pr}+2s_2)(2\ga+3).
\end{equation} 

{\fl Finally, we use the fact} that $v_1=t_1-1$ together with the estimate for $\wt{I}_V-\wt{I}_T$ obtained in \eqref{last_estimate2} to deduce that
\begin{equation}\label{I_V_minus_I_T}
\wt{I}_V-\wt{I}_T\geq r(v_1+2v_2)-(3v_1-u_1)-3\sum_{i=t_2-r-1}^{t_2-1}i-\sum_{i=1}^{4}n_i.
\end{equation}
Summing \eqref{I_S_minus_I_V} and \eqref{I_V_minus_I_T} yields
\[
\wt{I}_S-\wt{I}_T
\geq r(v_1+2v_2)-(3v_1-u_1)-3\sum_{i=t_2-r-1}^{t_2-1}i-\sum_{i=1}^{4}n_i -(s_1^{\pr}+2s_2)(2\ga+3)
\]
which for the same reasons as before is strictly positive whenever $g\gg\ga$.
\end{proof}

According to Proposition~\ref{maxT1v2}, any $(3,\ga)$-hyperelliptic semigroup of genus $g$ and maximal weight will be a semigroup $\rm S$ 
for which $\#{\rm S}_{(u^{\pr}_1)_3}=\#{\rm T}_{(u_1)_3}$ is
maximal, where $\rm T$ is as in Proposition~\ref{maxT1v2}.

Given such a maximal-weight semigroup ${\rm S}$, let $u_1^{\pr}= u_1({\rm S})$ as usual. Note that if $u_1^{\pr}>u_1$, then either ${\rm S}_{(u^{\pr}_1)_3}\subsetneq{{\rm T}_{(u_1)_3}}$ whenever $(u_1^{\pr})_3=(u_1)_3$; and otherwise ${\rm S}_{(u^{\pr}_1)_3}\subset{{\rm V}_{(u^{\pr\pr}_1)_3}}$ 
where ${\rm V}$ is a semigroup constructed in the proof of Proposition~\ref{maxT1v2}. Either way we have $\wt{I}_{S}>\wt{I}_{T}$, which is absurd. So we conclude that $u_1^{\pr}\leq u_1$.


\subsection{Step 2} To identify $(3,\ga)$-hyperelliptic semigroups of maximal weight, we will compare the inflection of the semigroup ${\rm T}$ of Lemma~\ref{T1form} with the inflection of an arbitrary semigroup ${\rm S}$ that satisfies the following two conditions:   
    \begin{enumerate}
        \item $u_1({\rm S})\leq u_1=g-\ga+1+\lfloor\frac{(g-\ga)_3}{2}\rfloor$; and\label{condition1}
        \item $\#{\rm S}_{(u^{\pr}_1)_3}=\#{\rm T}_{(u_1)_3}$.\label{condition2} 
    \end{enumerate}
\medskip
Let $u_1^{\pr}= u_1({\rm S})$ as usual. When $\ga=0$, the unique semigroup that satisfies conditions \ref{condition1} and \ref{condition2} is the semigroup ${\rm T}$ of Proposition~\ref{maxT1v2}. Indeed, any semigroup  ${\rm S}$ that satisfies condition \ref{condition2} is such that ${\rm S}_{3-(u_1^{\pr})_3}={\rm T}_{3-(u_1)_3}=\emptyset$. Moreover, any ${\rm S} \neq {\rm T}$ that satisfies both conditions 1 and 2 has $u_1^{\pr}$ strictly less than
$u_1({\rm T})=g+1+\lceil\frac{(g)_3}{2}\rceil$. Since $(u_1)_3\neq 0$, it would follow that $u_1^{\pr}\leq g$, and so $2u_1^{\pr}$ would belong to ${\rm S}_{3-(u_1^{\pr})_3}$, contradiction. 
\medskip

{\fl For} $\ga>0$ we will use the following result.


\begin{claim}\label{aux3}
Let $g$ be a positive integer and let $T=\{x_0,x_1,\dots\,x_k\}$ be an increasing 
arithmetic sequence of nonnegative integers, 
where $x_k$ is the maximal integer for which $x_0+x_k\leq 2g$. Now suppose that $S=T\setminus \{x_i\}$ for some $i \neq 0$. Then $\#(S+S)\cap[2g]=\#(T+T)\cap[2g]-1$ (resp., $\#(S+S)\cap[2g]=\#(T+T)\cap[2g]$) if and only if $i=1$ (resp, $i>1$).
\end{claim}

\begin{proof}[Proof of Claim~\ref{aux3}] To begin, note that $(T+T)\cap[2g]=\{2x_0,x_0+x_1,...,x_0+x_k\}$. It follows that when $i=1$, we have 
$(S+S)\cap[2g]=(T+T)\cap[2g]\setminus{(x_0 +x_1)}$. On the other hand, when $i>1$, we have $(S+S)\cap[2g]=(T+T)\cap[2g]$, because $\{2x_0,x_0+x_1,...,x_0+x_k\}\setminus\{x_0+x_i\}$ belongs to $(S+S)\cap[2g]$, and by construction $x_0+x_i=x_1+x_{i-1}$, so in particular $x_0+x_i$ belongs to $(S+S)\cap[2g]$.
\end{proof}

We now filter possibilities according to 1) whether or not $(g-\ga)_3=2$; and 2) whether or not $u_1$ and $u_1^{\pr}$ have equal residues.

{\fl \bf Case 1: $(g-\ga)_3 \neq 2$.} As in the proof of Lemma~\ref{T1form}, we have $u_1=g-\ga+1$, $\#{\rm T}_{(u_1)_3}=\frac{g+\ga+2}{3}- \frac{(g+\ga-1)_3}{3}$ and $\#{\rm T}_{3-(u_1)_3}=\frac{2\ga+1}{3}-\frac{(2\ga-2)_3}{3}$. 
More precisely, we have
    {\small
    \[
    {\rm T}_{(u_1)_3}=\{u_1,u_1+3,\dots,2g-(g+\ga-1)_3\} \text{ and } {\rm T}_{3-(u_1)_3}=\{2u_1,2u_1+3,\dots,2g-(2\ga+1)_3\}.
    \]
}
{\fl Now} suppose that $(u_1^{\pr})_3=(u_1)_3$. Then $u^{\pr}_1=g-\ga+1-3e$, for some $e \geq 0$. Note that $u^{\pr}_1+u\leq 2g$ for all $u \leq g+\ga-1+3e$, so that $u+u_1^{\pr} \in {\rm S}_{3-(u_1^{\pr})_3}$ for all $u \in {\rm S}_{u_1^{\pr}} \cap [g+\ga-1+3e-(2\ga-2)_3]$. Now 
{\small
\[
\begin{split}
\#{\rm S}_{(u^{\pr}_1)_3}\cap[g+\ga+2+3e-(2\ga-2)_3,2g-(g+\ga-1)_3]&\leq\frac{g-\ga+1}{3}- \frac{(g-\ga+1)_3}{3}+ \frac{(2\ga-2)_3}{3}-e \\
\text{ i.e., }k:=\#{\rm S}_{(u^{\pr}_1)_3}\cap[g+\ga-1+3e-(2\ga-2)_3]&\geq \frac{2\ga+1}{3}- \frac{(2\ga-2)_3}{3}+e.
\end{split}
\]
}
\hspace{-3pt}That is to say, we have ${\rm S}_{(u^{\pr}_1)_3}\cap[g+\ga-1+3e-(2\ga-2)_3]=\{u^{\pr}_1,u^{\pr}_1+a_1,\dots,u^{\pr}_1+a_{k-1}\}$ for distinct positive integers $a_1, \dots, a_{k-1}$. It follows that $2u^{\pr}_1,2u^{\pr}_1+a_1,\dots,2u^{\pr}_1+a_{k-1}$ are $k$ distinct elements that belong to ${\rm S}_{3-(u^{\pr}_1)_3}$. The obvious restriction $k \leq \frac{2\ga+1}{3}- \frac{(2\ga-2)_3}{3}$ now forces $e=0$ and, consequently, ${\rm S}={\rm T}$. 

{\fl Suppose} now that $(u_1^{\pr})_3\neq(u_1)_3$. Then $u^{\pr}_1=g-\ga-2+(u_1)_3-3e$, for some $e\geq0$. In particular, we have $u^{\pr}_1+u\leq 2g$ for all $u\leq g+\ga+2-(u_1)_3+3e$, so $u+u_1^{\pr} \in {\rm S}_{3-(u_1^{\pr})_3}$ for all $u \in {\rm S}_{(u_1^{\pr})_3} \cap [g+\ga+2-(g-\ga+1)_3+3e- (g+\ga-1)_3]$. 

{\fl Now set}
\[
\begin{split}
k_1:&=\#{\rm S}_{(u^{\pr}_1)_3}\cap[g+\ga+5-(g-\ga+1)_3+ 3e-(g+\ga-1)_3,2g-(2\ga+1)_3] \text{ and}\\
k_2:&=\#{\rm S}_{(u^{\pr}_1)_3}\cap[g+\ga+2-(g-\ga+1)_3+ 3e-(g+\ga-1)_3].
\end{split}
\]
We have
\[
\begin{split}    
k_1&\leq\frac{g-\ga-2}{3}+ \frac{(g-\ga+1)_3}{3}-\frac{(2\ga+1)_3}{3}+ \frac{(g+\ga-1)_3}{3}-e, \text{ and therefore}\\
k_2&\geq \frac{2\ga+4}{3}- \frac{2(g+\ga-1)_3}{3}-\frac{(g-\ga+1)_3}{3}+ \frac{(2\ga+1)_3}{3}+e.
\end{split}
\]

{\fl It follows} immediately that $\#{\rm S}_{3-(u^{\pr}_1)_3} \geq k_2$, which in turn means that
$\varphi_1(g,\ga)+ 3e \leq 0$,
where 
\[
\varphi_1(g,\ga):=3+ 2(2\ga-2)_3- 2(g+\ga-1)_3- (g-\ga+1)_3.
\]
It is easy to check that $\varphi_1$ is strictly positive unless $((g)_3,(\ga)_3) \in \{(0,0), (1,1)\}$, in which case $\varphi_1=0$; or $((g)_3,(\ga)_3)=(2,1)$, in which case $\varphi_1=-3$. We treat each of these possibilities in turn.

\medskip
{\fl If} $((g)_3,(\ga)_3) \in \{(0,0), (1,1)\}$, then $e=0$ and thus $u_1^{\pr}=u_1-2$. Our construction forces 
\[
\begin{split}
{\rm S}_{(u^{\pr}_1)_3} \cap [g+\ga+2-(g-\ga+1)_3+ 3e-(g+\ga-1)_3]&= \{u^{\pr}_1,u^{\pr}_1+a_1,\dots,u^{\pr}_1+a_{k_2-1}\} \text{ and }\\
({\rm S}_{(u^{\pr}_1)_3}+{\rm S}_{(u^{\pr}_1)_3})\cap[2g]&=\{2u^{\pr}_1,2u^{\pr}_1+a_1,\dots,2u^{\pr}_1+a_{k_2-1}\}
\end{split}
\]
to be, up to insertions of singletons, arithmetic sequences. According to Claim~\ref{aux3}, this is only possible if $u_1^{\pr}+3$ is the only gap in ${\rm S}_{(u^{\pr}_1)_3} \cap [g+\ga+2-(g-\ga+1)_3+ 3e-(g+\ga-1)_3]$ 
(by which we mean that the elements larger than $u_1^{\pr}+3$ form an arithmetic sequence). The unique semigroup $\rm S$ that satisfies this requirement, in addition to the two conditions \ref{condition1} and \ref{condition2} above, is such that $u_1^{\pr}=u_1-2$, 
{\small
\[
{\rm S}_{(u^{\pr}_1)_3}=\{u_1-2,u_1+4,u_1+7,\dots,2g-(2\ga+1)_3\} \text{ and }{\rm S}_{3-(u^{\pr}_1)_3}=\{2u_1-4,2u_1+2,2u_1+5,\dots,2g-(2\ga+2)_3\}.
\]
}
{\fl We} now compare $\wt{I}_{\rm T}$ and $\wt{I}_{\rm S}$. For this purpose, note that from the second onwards all $\frac{g+\ga+2}{3}-\frac{(g+\ga-1)_3}{3}-1$ (resp. $\frac{2\ga+1}{3}-\frac{(2\ga-2)_3}{3}-1$) successive differences between elements of ${\rm S}_{(u^{\pr}_1)_3}$ and  ${\rm T}_{(u_1)_3}$ (resp., ${\rm S}_{3-(u^{\pr}_1)_3}$ and ${\rm T}_{3-(u_1)_3}$) are $1$ (resp., $-1$). It follows that
{\small
\[
\wt{I}_{{\rm S}_{(u^{\pr}_1)_3}}-\wt{I}_{{\rm T}_{(u_1)_3}}=\frac{g+\ga+2}{3}-\frac{(g+\ga-1)_3}{3}-3\text{ and } \wt{I}_{{\rm S}_{3-(u^{\pr}_1)_3}}-\wt{I}_{{\rm T}_{3-(u_1)_3}}=-\frac{2\ga+1}{3}+\frac{(2\ga-2)_3}{3}-3.
\]
}
{\fl Summing} we obtain $\wt{I}_{{\rm S}}-\wt{I}_{{\rm T}}=\frac{g-\ga+1}{3}-\frac{(g+\ga-1)_3}{3}+\frac{(2\ga-2)_3}{3}-6$, which is positive whenever g is large relative to $\ga$.
\medskip
{\fl Suppose} now that $((g)_3,(\ga)_3)=(2,1)$. Then either $e=0$ or $e=1$. 
Say that $e=0$; then $u_1^{\pr}=g-\ga=u_1-1$. In particular, ${\rm S}_{(u^{\pr}_1)_3} \sub ((u_1-1)+3\mathbb{N})\cap[2g]$, and it is easy to see that in fact ${\rm S}_{(u^{\pr}_1)_3}$ is obtained from the arithmetic sequence $((u_1-1)+3\mathbb{N})\cap[2g]= \{u_1-1, u_1+2,\dots, 2g\}$ by deleting a single element $(u_1-1)+ 3\ell$, for some $\ell \geq 1$. Claim~\ref{aux3} explains how to classify all of the corresponding semigroups ${\rm S}$ satisfying conditions \ref{condition1} and \ref{condition2} above. The first possibility is that $\ell>1$, in which case we have 
\[
{\rm S}_{3-(u_1^{\pr})_3}= ({\rm S}_{(u^{\pr}_1)_3}+{\rm S}_{(u^{\pr}_1)_3})= \{2u_1-2,2u_1+1,\dots,2g-2\}.
\]
Here $\ell<2\ga$, and more precise restrictions on $\ell$ arise from the structure of ${\rm S}_0$, which is as-yet unspecified. The second possibility is that $\ell=1$, in which case $({\rm S}_{(u^{\pr}_1)_3}+{\rm S}_{(u^{\pr}_1)_3})$ is properly contained in ${\rm S}_{3-(u_1^{\pr})_3}$ and in order to minimize inflection we may take
\[
{\rm S}_{3-(u_1^{\pr})_3}=\{2u_1-2,2u_1+4,2u_1+7,\dots,2g-2\} \sqcup \{2u_1-2-n_1\}
\]
where as usual $n_1$ denotes the smallest nonzero element of ${\rm T}_0$.
Let ${\rm S}^1$ (resp., ${\rm S}^2$) denote the inflection-minimizing semigroup when $\ell>1$ (resp., $\ell=1$). We compute
\[
\wt{I}_{{\rm S}^1}-\wt{I}_{{\rm T}}=\frac{2g-2\ga-2}{3}-3\ell \text{ and }
\wt{I}_{{\rm S}^2}-\wt{I}_{{\rm T}}=
\frac{2g-2\ga+1}{3}-5-n_1.
\]


\medskip
{\fl Finally}, suppose that $((g)_3,(\ga)_3)=(2,1)$ and $e=1$. Then $u_1^{\pr}=g-\ga-3=u_1-4$, and ${\rm S}_{(u_1^{\pr})_3}$ (resp., ${\rm S}_{3-(u_1^{\pr})_3}$) is obtained from the arithmetic sequence $(u_1-4)+3\mb{N}) \cap [2g]$ (resp., $(2u_1-8)+3\mb{N}) \cap [2g]$) it generates via two deletions. However, an argument analogous to that used to prove Claim~\ref{aux3} shows this is only possible if those deletions from $(u_1-4)+3\mb{N}) \cap [2g]$ are precisely the second and third smallest elements of the arithmetic sequence. Thus
{\small
\[
{\rm S}_{(u_1^{\pr})_3}= \{u_1-4,u_1+5,u_1+8,\dots, 2g\} \text{ and } {\rm S}_{3-(u_1^{\pr})_3}= \{2u_1-8, 2u_1+1, 2u_1+4, \dots, 2g-1\}.
\]
}
\hspace{-3pt}It follows that $\wt{I}_{{\rm S}}-\wt{I}_{{\rm T}}=\frac{2g-2\ga+1}{3}-13$ is positive whenever $g \gg \ga$.

    \medskip
    {\fl \bf Case 2: $(g-\ga)_3=2$.} In this case, we have $u_1=g-\ga+2$, 
    \[
    \begin{split}
    \#{\rm T}_{(u_1)_3}&=\frac{2g-u_1}{3}+1-\frac{(2g-u_1)_3}{3}=\frac{g+\ga+1}{3}- \frac{(2\ga)_3}{3}, \text{ and } \\
    \#{\rm T}_{3-(u_1)_3}&=\frac{2g-2u_1}{3}+2-\frac{(2g-2u_1)_3}{3}=\frac{2\ga+2}{3}- \frac{(2\ga-1)_3}{3}.
    \end{split}
    \]
    More precisely, according to Remark~\ref{completeT2}, we have
    {\small
    \[
    {\rm T}_{(u_1)_3}=\{u_1,u_1+3,\dots,2g-(2\ga)_3\} \text{ and } {\rm T}_{3-(u_1)_3}=\{2u_1-n_1,2u_1,2u_1+3,\dots,2g-(2\ga-1)_3\}
    \]
}
{\fl Now} suppose that $(u_1^{\pr})_3=(u_1)_3$. Then $u^{\pr}_1=g-\ga+2-3e$, for some $e\geq0$. Note that $u^{\pr}_1+u\leq 2g$, 
for all $u\leq g+\ga-2+3e$; it follows that $u^{\pr}_1+u \in {\rm S}_{3-(u_1^{\pr})_3}$ for all $u \in {\rm S}_{(u_1^{\pr})_3} \cap [g+\ga-4+3e]$. Now
\[
\#{\rm S}_{(u^{\pr}_1)_3}\cap[g+\ga-1+3e,2g-(2\ga)_3]\leq \frac{g-\ga+4}{3}- \frac{(2\ga)_3}{3}-e
\]
which in turn implies that
\[
k:=\#{\rm S}_{(u^{\pr}_1)_3}\cap[g+\ga-4+3e] \geq \frac{2\ga}{3}+ e-1.
\]

{\fl The} obvious restriction $k \leq \#{\rm T}_{3-(u_1)_3}$ now forces $\varphi_2(\ga)+ 3e \leq 0$, where
\[
\varphi_2(\ga):=(2\ga-1)_3-5.
\]

{\fl It} follows immediately that $e \leq 1$. If $e=0$, then ${\rm S}={\rm T}$; so we may assume e=1 without loss of generality. Writing ${\rm S}_{(u_1^{\pr})_3}= \{u_1^{\pr}, u_1^{\pr}+a_1,\dots, u_1^{\pr}+a_{k-1}\}$ as usual, our construction shows that ${\rm S}_{(u_1^{\pr})_3}$ and $({\rm S}_{(u^{\pr}_1)_3}+{\rm S}_{(u^{\pr}_1)_3})\cap[2g]=\{2u^{\pr}_1,2u^{\pr}_1+a_1,\dots,2u^{\pr}_1+a_{k-1}\}$ are, up to insertions by singletons, arithmetic sequences. 
Claim~\ref{aux3} now implies $u_1^{\pr}+3$ is the unique gap in ${\rm S}_{(u^{\pr}_1)_3}\cap[g+\ga-1+3e]$. 
The only semigroup $\rm S$ that satisfies this requirement, in addition to the two conditions \ref{condition1} and \ref{condition2} above, is such that $u_1^{\pr}=u_1-3$, 
{\small
\[
{\rm S}_{(u^{\pr}_1)_3}=\{u_1-3,u_1+3,u_1+6,\dots,2g-(2\ga)_3\} \text{ and } {\rm S}_{3-(u^{\pr}_1)_3}=\{2u_1-6,2u_1,2u_1+3,\dots,2g-(2\ga-1)_3\}.
\]
}
\hspace{-3pt}It is now easy to see that 
$\wt{I}_{{\rm S}}-\wt{I}_{{\rm T}}=(3u_1-9)-(3u_1-n_1)=n_1-9$.

\medskip
{\fl Finally}, suppose that $(u_1^{\pr})_3\neq(u_1)_3$; then $u^{\pr}_1=g-\ga-3e$ for some $e\geq0$. Note that $u^{\pr}_1+u\leq 2g$ for all $u\leq g+\ga+3e$. In particular, we have $u+u_1^{\pr} \in {\rm S}_{3-(u_1^{\pr})_3}$ for all $u \in {\rm S}_{(u_1^{\pr})_3} \cap [g+\ga+3e-(2\ga)_3]$. Now
\[
\begin{split}    
\#{\rm S}_{(u^{\pr}_1)_3}\cap[g+\ga+3e-(2\ga)_3+3,2g-(2\ga-1)_3]&\leq \frac{g-\ga}{3} -\frac{(2\ga-1)_3}{3}+ \frac{(2\ga)_3}{3}-e, \text{ i.e.,}\\
k:=\#{\rm S}_{(u^{\pr}_1)_3}\cap[g+\ga+3e-(2\ga)_3]&\geq \frac{2\ga+1}{3}-\frac{2(2\ga)_3}{3}+\frac{(2\ga-1)_3}{3}+e.
\end{split}
\]
The obvious restriction $k \leq \#{\rm T}_{3-(u_1)_3}$ now forces $\varphi_3(\ga)+ 3e \leq 0$, where
\begin{equation}\label{phi3_restriction}
\varphi_3(\ga):=2(2\ga-1)_3-2(2\ga)_3-1.
\end{equation}
It follows immediately that $e \leq 1$. Say that $e=0$; then $u^{\pr}=u_1-2$. If further $(\ga)_3=0$ and condition~\ref{condition2} holds, then ${\rm S}_{(u_1^{\pr})_3}=((u_1-2)+3\mathbb{N})\cap[2g]$, which contradicts the minimality property of $u_1$ established in Lemma~\ref{T1form}. On the other hand, if $(\ga)_3 \in \{1,2\}$, then ${\rm S}_{(u_1^{\pr})_3}$ is obtained by deleting one element from $((u_1-2)+3\mb{N}) \cap [2g]$, while ${\rm S}_{3-(u_1^{\pr})_3})=((2u_1-4)+3\mb{N}) \cap [2g]$). 
According to Claim~\ref{aux3}, this means that ${\rm S}_{(u_1^{\pr})_3}= ((u_1-2)+3\mb{N}) \cap [2g] \setminus \{(u_1-2)+3\ell\}$ for some positive integer $1 < \ell < 2\ga$. We compute 
\[
\wt{I}_{\rm S}-\wt{I}_{\rm T}= \frac{g-7\ga-1}{3}-3\ell+ \frac{2(2\ga)}{3}-\frac{4(2\ga-1)_3}{3}
\]
which is strictly positive whenever $g \gg \ga$.

\medskip
Finally, say that $e=1$; the nonpositivity restriction \eqref{phi3_restriction} forces $(\ga)_3 \in \{1,2\}$. In either case, we have $u_1^{\pr}=u_1-5$, and ${\rm S}_{(u_1^{\pr})_3}$ (resp., ${\rm S}_{3-(u_1^{\pr})_3}$) is obtained from $((u_1-5)+3\mb{N}) \cap [2g]$ (resp., $((2u_1-10)+3\mb{N}) \cap [2g]$ via {\it two} deletions. It follows just as in the analysis of the case of $((g)_3,(\ga)_3)=(2,1)$ and $e=1$ in Case 1 above that the two deletions from $((u_1-5)+3\mb{N}) \cap [2g]$ are necessarily the second and third elements of that arithmetic sequence, i.e., that
{\small
\[
\begin{split}
{\rm S}_{(u_1^{\pr})_3}&=\{u_1-5,u_1+4,u_1+7,\dots,2g-(2\ga-1)_3\} \text{ and}\\
{\rm S}_{3-(u_1^{\pr})_3}&=\{2u_1-10, 2u_1-1, 2u_1+2, \dots, 2g-(2\ga)_3\}.
\end{split}
\]
}
{\fl We} compute $\wt{I}_{{\rm S}}-\wt{I}_{{\rm T}}=n_1-15+\frac{g-\ga-1}{3}-\frac{(2\ga)_3}{3}+ \frac{(2\ga-1)_3}{3}$. 

\subsection{Step 3} To conclude the proof of Theorem~\ref{maximal_weight_conjecture}, we use the following result.

\begin{prop}\label{chooseT0}
 Let ${\rm S}$ and ${\rm T}$ be semigroups of maximal weight in ${\rm H}^{{\rm S}_0}$ and ${\rm H}^{{\rm T}_0}$, respectively, and suppose that $6 \in {\rm T}_0$.  Then
\[ 
W_{\rm T} \geq W_{\rm S} \iff \sum_{s \in {\rm S}_0} s \leq \sum_{t \in {\rm T}_0} t.
\]
{\fl In} particular, ${\rm S}$ has maximal weight in $\rm H$ if and only if $6\in {\rm S}$.
\end{prop}
\begin{proof}
In light of our analysis of Cases 1 and 2 above, it suffices to show that $\wt{I}_{{\rm S}_0}- \wt{I}_{{\rm T}_0}$ is $g$-asymptotically larger than $n_1({\rm S})+{\text (constant)}$. On the other hand, each of ${\rm T}_0$ and ${\rm S}_0$ is equal to a (triple) dilation of a semigroup of genus $\ga$. Our assumption that ${\rm T}_0$ contains 6 means precisely that the genus-$\ga$ semigroup associated with ${\rm T}_0$ is hyperelliptic. Accordingly we must obtain a uniform lower bound for the difference $\wt{I}_{{\rm T}^1}-\wt{I}_{{\rm T}^2}$, where ${\rm T}^1$ and ${\rm T}^2$ are semigroups of genus $\ga>0$, ${\rm T}^2$ is hyperelliptic, and $n_1=n_1({\rm T}^1)>1$ is fixed. To do so, we may suppose that ${\rm T}^1$ is {\it minimally} inflected among all genus-$\ga$ semigroups ${\rm S}$ with $n_1(S)=n_1$. Graphically, this means that the Dyck path that encodes ${\rm T}^1$ stabilizes to a staircase with step-size 1 after precisely $(n_1-1)$ initial vertical moves upward, as in Figure 1. This, however, means precisely that the difference between the $i$th nonzero positive elements of ${\rm T}^1$ and ${\rm T}^2$ (where $n_i<n_j$ whenever $i<j$) satisfies $n_i({\rm T}^1)-n_i({\rm T}^2)=n_1-2$. It then follows immediately that $\wt{I}_{{\rm T}^1}-\wt{I}_{{\rm T}^2} \geq (g-1-\frac{n_1-1}{2})(n_1-2)$, i.e., that $\wt{I}_{{\rm S}_0}- \wt{I}_{{\rm T}_0} \geq 3(g-1-\frac{n_1-1}{2})(n_1-2)$.
\end{proof}

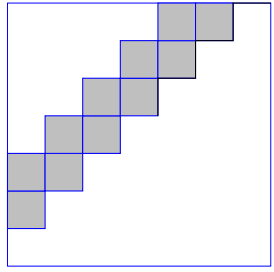
\begin{figure}
\begin{tikzpicture}[scale=0.50]
\draw[blue, very thin] (0,0) rectangle (7,7);
\filldraw[draw=blue, fill=lightgray] (0,1) rectangle (1,2);
\filldraw[draw=blue, fill=lightgray] (0,2) rectangle (1,3);

\filldraw[draw=blue, fill=lightgray] (1,2) rectangle (2,3);
\filldraw[draw=blue, fill=lightgray] (1,3) rectangle (2,4);

\filldraw[draw=blue, fill=lightgray] (2,3) rectangle (3,4);
\filldraw[draw=blue, fill=lightgray] (2,4) rectangle (3,5);

\filldraw[draw=blue, fill=lightgray] (3,4) rectangle (4,5);
\filldraw[draw=blue, fill=lightgray] (3,5) rectangle (4,6);

\filldraw[draw=blue, fill=lightgray] (4,5) rectangle (5,6);
\filldraw[draw=blue, fill=lightgray] (4,6) rectangle (5,7);

\filldraw[draw=blue, fill=lightgray] (5,6) rectangle (6,7);

\draw (4,4) -- (4,5) -- (5,5) -- (5,6) -- (6,6) -- (6,7) -- (7,7);

\end{tikzpicture}
\caption{The area of the shaded boxes computes the minimal inflection discrepancy $\wt{I}_{{\rm T}^1}-\wt{I}_{{\rm T}^2}$ in the proof of Proposition~\ref{chooseT0} when $g=7$ and $n_1({\rm T}^1)=4$.}
\end{figure}

\section{The geometric realization problem for $(N,\ga)$-hyperelliptic semigroups of maximal weight}\label{geometric_realization}

\begin{conj}\label{maximal_weight_conj}
Assume that $N \geq 3$ is prime, $(\frac{N-1}{2})|(g-\ga)$, and that 
$(\frac{2(g-\ga)}{N-1})_N \neq N-1$. Then there is a unique $(N,\ga)$-hyperelliptic semigroup of $g$-asymptotically maximal weight given by
\[
{\rm S}_N:=N \langle 2,2\ga+1 \rangle+ \bigg \langle \frac{2(g-\ga)}{N-1}+1, \frac{2(g-\ga)}{N-1}+N+1 \bigg \rangle.
\]
\end{conj}

\begin{rem}\label{S_N_facts}
${\rm S}_N$ is of genus $g$. Indeed, given $0 \leq j \leq N-1$, let ${\rm S}_{N,j} \sub {\rm S}_N \cap [2g]$ denote the critical subset of elements with $N$-residue $j$. By construction, we have
    \[
    {\rm S}_{N,0}= \{2N, 4N, \dots, 2\ga N\} \text{, while } {\rm S}_{N,(ju_1)_N}= 
    \{ju_1, ju_1+N, \dots\}
    \]
    for all $1 \leq j \leq N-1$, where $u_1=\frac{2(g-\ga)}{N-1}+1$. Our expectation that ${\rm S}_N$ is of maximal weight is predicated on the facts that
    \begin{itemize}
    \item ${\rm S}_N$ contains $N\langle 2, 2\ga+1 \rangle$, i.e. its subset of residue-zero elements is of maximal weight;
    \item its subsets of fixed positive residue $j>0$ are all saturated, in the sense that all gaps of residue $j$ residue are strictly smaller than any nonzero elements belonging to ${\rm S}_{N,j}$; and
    \item all insertions to ${\rm S}_N$ beyond the residue-zero subset $N\langle 2, 2\ga+1 \rangle$ are generated by insertions from a single positive residue class.
    \end{itemize}
    The corresponding sets of gaps ${\rm G}_{N,j}$ are given by
    \[
    {\rm G}_{N,0}= \{N,3N, \dots, (2\ga-1)N\} \text{ and } {\rm G}_{N,(ju_1)_N}= \{(ju_1)_N, (ju_1)_N+N, \dots, ju_1-N\}
    \]
    for all $1 \leq j \leq N-1$, respectively. In particular, this means that for positive values of $j$, the gap set ${\rm G}_{N,j}$ is of cardinality 
    $g_{N,(ju_1)_N}:= \frac{ju_1-(ju_1)_N}{N}$, and consequently the complement of ${\rm S}_N$ in $\mb{N}$ is of cardinality
    \[
    \ga+ \sum_{j=1}^{N-1} g_{N,(ju_1)_N}= \ga+ \frac{2(g-\ga)}{N(N-1)} \sum_{j=1}^{N-1} j= g.
    \]
    \end{rem}
    
    \begin{rem}
    The weight of ${\rm S}_N$ is calculated by
    \begin{equation}\label{w_N}
    w_N:= \sum_{j=1}^{(N-1)} \sum_{k=0}^{(g_{N,(ju_1)_N}-1)} ((ju_1)_N+kN) + N\ga^2- \binom{g+1}{2}.
    \end{equation}
    Pairing complementary summands in arithmetic progressions, we see that the $j$th contribution to the double sum in \eqref{w_N} is precisely
    \[
   \frac{(ju_1)-(ju_1)_N}{N} \cdot \frac{(ju_1)+(ju_1)_N-N}{2}= \frac{u_1^2}{2N} j^2- \frac{1}{2N} ((ju_1)_N)^2- \frac{u_1}{2} j+ \frac{1}{2} (ju_1)_N.
    \]
    Consequently, we may rewrite the double sum as
    {\small
    \[
    \frac{u_1^2}{2N} \sum_{j=1}^{N-1}j^2- \frac{1}{2N} \sum_{j=1}^{N-1}((ju_1)_N)^2- \frac{u_1}{2} \sum_{j=1}^{N-1}j+ \frac{1}{2} \sum_{j=1}^{N-1}(ju_1)_N= \frac{u_1^2-1}{2N} \sum_{j=1}^{N-1}j^2- \frac{u_1-1}{2} \sum_{j=1}^{N-1}j.
    \]
    }

{\fl Here}
\[
\sum_{j=1}^{N-1} j^2= \frac{(N-1)N(2N-1)}{6}, \text{ while } \sum_{j=1}^{N-1} j= \frac{(N-1)N}{2}.
\]
We deduce that
\[
w_N= \frac{(2N-1)}{3(N-1)}(g-\ga)^2+ \frac{N-2}{6}(g-\ga)+ N\ga^2- \binom{g+1}{2}.
\]
\end{rem}

We now compute the weight of a semigroup (implicitly) considered by Carvalho--Torres in \cite{CT}, namely ${\rm S}_N^{\ast}:= N \langle 2,2\ga+1 \rangle+ \langle \frac{2}{N-1}g- \frac{2N}{N-1}\ga+1 \rangle$, whenever $g$ and $\ga$ are both multiples of $\frac{N-1}{2}$ and $(\frac{2g}{N-1})_N \neq N-1$.

\medskip
To this end, first note that the critical subsets of ${\rm S}^{\ast}_N$ of fixed residues are described by
\[
{\rm S}^{\ast}_{N,0}= \{2N, 4N, \dots, 2\ga N, (2\ga+1)N, \dots\}, \text{ while } {\rm S}^{\ast}_{N,(ju_1)_N}= ju_1+ \ov{{\rm S}^{\ast}}_{N,0}
\]
for all $1 \leq j \leq N-1$, where $u_1=\frac{2}{N-1}g- \frac{2N}{N-1}\ga+1$. The corresponding gap sets are
{\small
\[
{\rm G}^{\ast}_{N,0}= \{N, 3N, \dots, (2\ga-1) N\} \text{ and } {\rm G}^{\ast}_{N,(ju_1)_N}= \{(ju_1)_N, (ju_1)_N+ N, \dots, ju_1-N\} \sqcup (ju_1+{\rm G}^{\ast}_{N,0})
\]
}
for all $1 \leq j \leq N-1$. It follows that the weight of ${\rm S}_N^{\ast}$ is computed by
{\small
\[
\begin{split}
w_N^{\ast}:&= 
\frac{u_1^2-1}{2N} \sum_{j=1}^{N-1}j^2- \frac{u_1-1}{2} \sum_{j=1}^{N-1}j+ u_1 \ga \sum_{j=1}^{N-1}j  +2N\ga^2- \binom{g+1}{2} \\
&=\frac{(2N-1)}{3(N-1)}(g-N\ga)^2+ \frac{(N-2)}{6}(g-N\ga)+ (g-N\ga)N\ga+ 2N\ga^2+ \binom{N}{2}\ga- \binom{g+1}{2}.
\end{split}
\]
}
The crucial point here is that the $g\ga$-coefficient of $w_N$ is strictly greater than that of $w_N^{\ast}$, and this means that for any fixed choice of $\ga$, $w_N$ is $g$-asymptotically larger than $w_N^{\ast}$.

\begin{prob} Find an $N$-fold pointed cover $\pi: (C,p) \rightarrow (B,q)$ totally ramified in $p$ and such that ${\rm S}(C,p)=S_N$.
\end{prob}

Conjecture~\ref{maximal_weight_conjecture} predicts, in particular, that if such a cover exists, then when the target is of positive genus the source curve is an $N$-fold cover of a hyperelliptic curve marked in a Weierstrass point. It is worth noting that the case in which $N=3$ and the target is $\mb{P}^1$ is well-understood in terms of the geometry of the canonical embedding of the trigonal source curve; see \cite{Co,SV,BS}.

\subsection{Buchweitz' criterion and the realizability of ${\rm S}_N$ and ${\rm S}^{\ast}_N$}
Determining necessary and sufficient combinatorial conditions for when a given numerical semigroup ${\rm S}$ arises as the Weierstrass semigroup ${\rm S}(C,p)$ of a pointed curve is an old problem of Hurwitz, and remains very much open. To date, most examples of non-realizable numerical semigroups have been constructed using {\it Buchweitz' criterion} \cite{Bu}, which states that any realizable numerical semigroup ${\rm S}$ satisfies
\begin{equation}\label{buchweitz_criterion}
\# n({\rm G}) \leq (2n-1)(g-1)
\end{equation}
for all $n \geq 2$, where ${\rm G}=\mb{N} \setminus {\rm S}$ and $n({\rm G})$ denotes the set of $n$-fold sums of elements in ${\rm G}$. Buchweitz' criterion is an immediate consequence of the facts that whenever ${\rm S}={\rm S}(C,p)$,
\begin{itemize}
    \item The set of gaps ${\rm G}$ is in bijection with the set of $p$-vanishing orders of holomorphic differentials (i.e., 1-forms) on $C$;
    \item The set $n({\rm G})$ is naturally a subset of the set of $p$-vanishing orders of the space $H^0(C,K_C^{\otimes n})$ holomorphic $n$-forms on $C$; and
    \item $h^0(C,K_C^{\otimes n})= (2n-1)(g-1)$.
\end{itemize}

\begin{prop}\label{buchweitz_for_S_N}
Whenever $\ga>0$, the numerical semigroup ${\rm S}_N$ satisfies Buchweitz' criterion~\eqref{buchweitz_criterion} for every prime $N \geq 3$.
\end{prop}

\begin{proof}
Let ${\rm G}_N:= {\rm G}({\rm S}_N)$, and as in Remark~\ref{S_N_facts} let ${\rm G}_{N,j} \sub {\rm G}_N$ denote the subset of gaps with $N$-residue equal to $j$, $j=0, \dots, N-1$. It is easy to see that $2({\rm G})= {\rm G}+{\rm G}$ is in fact the disjoint union of the {\it completions} $\overline{{\rm G}_{N,((N-1)u_1)_N}+ {\rm G}_{N,(ju_1)_N}}$ of the (saturated) sets ${\rm G}_{N,((N-1)u_1)_N}+ {\rm G}_{N,(ju_1)_N}$, $j=0, \dots, N-1$, where we define the completion by $\overline{a+ [b,c]}:= a+ [0,c]$ whenever $a,b,c \in \mb{N}$. Note that
\[
{\rm G}_{N,0}=N\{1,3,\dots,2\ga-1\}, \text{ while } {\rm G}_{N,(ju_1)_N}=(ju_1)_N+ N\{0,1,\dots, \frac{ju_1-(ju_1)_N}{N}-1\}
\]
for all $j=1,\dots,N-1$; it follows that
\[
\begin{split}
\# \overline{{\rm G}_{N,((N-1)u_1)_N}+ {\rm G}_{N,0}}&= \frac{(N-1)u_1-((N-1)u_1)_N}{N}+ 2\ga-1\text{, while}\\
\# \overline{{\rm G}_{N,((N-1)u_1)_N}+ {\rm G}_{N,(ju_1)_N}} &= \frac{(N-1+j)u_1-(((N-1)u_1)_N+(ju_1)_N)}{N}-1
\end{split}
\]
for every $j=1,\dots,N-1$. As $u_1=\frac{2(g-\ga)}{N-1}+1$, it follows that
\[
\begin{split}
\#2({\rm G}_N) &= \sum_{j=0}^{N-1} \frac{(N-1+j)(\frac{2(g-\ga)}{N-1}+1)}{N}- ((N-1)u_1)_N- \sum_{j=1}^n \frac{(ju_1)_N}{N}- N+2\ga\\
&= \sum_{j=0}^{N-1} \frac{(N-1+j)\cdot \frac{2(g-\ga)}{N-1}}{N}- ((N-1)u_1)_N+ 2\ga-1 \\
&= 3g-\ga - ((N-1)u_1)_N-1
\end{split}
\]
and \eqref{buchweitz_criterion} follows when $n=2$. 

\medskip
More generally, whenever $n \geq 2$, $n({\rm G})$ is the disjoint union of the completions $\overline{(n-1){\rm G}_{N,N-1}+ {\rm G}_{N,j}}$, $j=0,\dots,N-1$, and we have 
\[
\begin{split}
\# \overline{{\rm (n-1)G}_{N,N-1}+ {\rm G}_{N,0}} &=\frac{(n-1)(N-1)u_1-(n-1)((N-1)u_1)_N}{N}+ 2\ga-1, \text{ while }\\
\#\overline{{\rm (n-1)G}_{N,N-1}+ {\rm G}_{N,j}} &=\frac{((n-1)(N-1)+j)u_1-((n-1)((N-1)u_1)_N+(ju_1)_N)}{N}-1
\end{split}
\]
for every $j=1,\dots,N-1$. It follows that
{\small
\[
\begin{split}
\#n({\rm G}_N) &= \sum_{j=0}^{N-1} \frac{((n-1)(N-1)+j)(\frac{2(g-\ga)}{N-1}+1)}{N} - (n-1)((N-1)u_1)_N- \sum_{j=1}^n \frac{(ju_1)_N}{N}- N+2\ga\\
&= (2n-1)g- (2n-3)\ga- (n-1)((N-1)u_1)_N-1
\end{split}
\]
}
and \eqref{buchweitz_criterion} follows for {\it every} $n \geq 2$.
\end{proof}

The upshot of Proposition~\ref{buchweitz_for_S_N} is that the realizability of ${\rm S}_N$ is somewhat subtle, inasmuch as it requires determining whether the {\it additive} structure of ${\rm G}_{N}$ and of its subsets ${\rm G}_{N,j}$ is compatible with the multiplicative structure of the pluricanonical section ring $\oplus_{n=1}^{\infty} H^0(C,K_C^{\otimes n})$ of an $(N,\ga)$-hyperelliptic pointed curve $(C,p)$. This should be contrasted with Torres' construction \cite{To3} of non-realizable numerical semigroups built out of ``fake" $(N,\ga)$-hyperelliptic covers by leveraging Buchweitz' criterion on the {\it base} of the cover, where it fails; in our case, Buchweitz' criterion is met on both base and target.

\medskip
On the other hand, under appropriate numerological assumptions, the Carvalho--Torres semigroup ${\rm S}_N^{\ast}$ {\it is} realizable by (the Weierstrass semigroup of a totally ramified point of) a cover $\pi: (C,p) \rightarrow (B,q)$ of pointed curves. Indeed, using the results of \cite[Sec. 3]{Bul}, it is not hard to show that ${\rm S}_N^{\ast}$ is the Weierstrass semigroup of a ramification point $p$ of a {\it cyclic} cover of a hyperelliptic curve $B$ marked in a hyperelliptic Weierstrass point $q$, and branched along a (reduced) divisor $D$ linearly equivalent to $dq$, where $d= \frac{2g-2- N(2\ga-2)}{N-1}$. In particular, the maximal weight of a {\it geometrically realizable} $(N,\ga)$-hyperelliptic semigroup of genus $g$, subject to the numerical constraints of Conjecture~\ref{maximal_weight_conjecture}, is {\it at least} $w_N^{\ast}$. 


\end{document}